\documentclass[11pt]{amsart}
\usepackage{mathptmx}
\usepackage{amsfonts,amsmath,graphics,amssymb,amsthm,verbatim, mathtools}
\usepackage{pdfsync}
\usepackage{amsfonts,amsmath,graphicx}
\usepackage{amsmath}
\usepackage{amssymb}
\usepackage{amsthm}
\usepackage{textcomp}
\usepackage{graphics}
\usepackage{geometry}
\newtheorem{theorem}{Theorem}[section]
\newtheorem{corollary}[theorem]{Corollary}
\newtheorem*{main}{Main Theorem}
\newtheorem{lemma}[theorem]{Lemma}
\newtheorem{proposition}[theorem]{Proposition}

\theoremstyle{definition}
\newtheorem{defn}[theorem]{Definition}
\newtheorem{remark}[theorem]{Remark}

\theoremstyle{remark}
\newtheorem{step}{Step}

\newcommand{\bpr}{\begin{proof}\hspace{3pt}}
\newcommand{\epr}{\end{proof}}
\newcommand{\lb}{\left(}
\newcommand{\rb}{\right)}
\newcommand{\curl}{\nabla \times}

\newcommand{\N}{\mathbb{N}}

\newcommand{\R}{\mathbb{R}}
\newcommand{\C}{\mathbb{C}}
\newcommand{\eps}{\epsilon}
\newcommand{\pa}{\partial}
\newcommand{\p}{\partial}

\begin{document}

\title{Approximate Isotropic Cloak for the Maxwell equations}
\author{Tuhin Ghosh, Ashwin Tarikere}

\address{Institute for Advanced Study, The Hong Kong University of Science and Technology  }
\email{iasghosh@ust.hk}
\address{Department of Mathematics, University of Washington}
\email{ashwin91@math.washington.edu}


\date{\today}


\maketitle

\begin{abstract}
We construct a regular isotropic approximate cloak for the Maxwell system of equations. The method of transformation optics has enabled the design of electromagnetic parameters that cloak a region from external observation. However, these constructions are singular and anisotropic, making practical implementation difficult. Thus, regular approximations to these cloaks have been constructed that cloak a given region to any desired degree of accuracy. In this paper, we show how to construct isotropic approximations to these regularized cloaks using homogenization techniques, so that one obtains cloaking of arbitrary accuracy with \emph{regular and isotropic} parameters.
\end{abstract}

\section{Introduction}

A region of space is said to be cloaked if its contents are indistinguishable from empty space by external measurements. The idea of invisibility cloaks has fascinated human beings throughout history, and in recent years there has been a wave of serious theoretical proposals for such a cloak that have gained considerable interest in mathematics and the broader scientific community \cite{GKLU3}. The particular approach to cloaking that has received the most attention is that of \emph{transformation optics}, the use of transformation rules for the material properties of optics under changes of coordinates: the index of refraction $n(x)$ for scalar optics, governed by the Helmholtz equation, and the electric permittivity $\epsilon(x)$ and magnetic permeability $\mu(x)$ for vector optics, as described by Maxwell's equations.The transformation optics approach to cloaking for the conductivity equation was first introduced by Greenleaf, Lassas and Uhlmann \cite{GKLU0,GKLU1} in 2003, in the context of providing counter-examples to the Calder\'{o}n Problem in dimensions $N \geq 3$. Three years later, using the same singular change of coordinates, Pendry, Schurig and Smith \cite{PSS} gave a prescription for values of $\epsilon(x)$ and $\mu(x)$ giving a cloaking device for electromagnetic waves. Also in 2006, Leonhardt \cite{Leon} gave a description, based on conformal mapping, of inhomogeneous indices of refraction $n(x)$ in two dimensions that would make light rays  go around a region and emerge on the other side as if it had passed through empty space.\\

 The method of transformation optics is based on the invariance properties of governing equations under changes of coordinates. One first selects a region $\Omega$ in space. For a fixed point $p \in \Omega$ and a bounded open subset $D \Subset \Omega$ containing $p$, let $F$ be a diffeomorphism from $\Omega \setminus \{p\}$ to $\Omega \setminus \overline{D}$. The material parameters in $\Omega \setminus \overline{D}$ are defined by \emph{pushing-forward} the material parameters of homogeneous space by the singular transformation $F$. Now using the transformation properties of the governing equation and a removable singularity theorem, it is shown that no matter what the material parameters are in $D$, the resulting acoustic or electromagnetic waves have the same scattering properties and boundary measurements as in the case where $\Omega$ is a homogeneous medium. We call $\Omega \setminus\overline{D}$ and $D$ the cloaking region and the cloaked region respectively. Similarly, cloaking devices based on blowing up a curve or a flat surface were considered in \cite{Wormhole} and \cite{Carpet} respectively, resulting in the so-called wormholes and carpet cloaks. Since these constructions are based on a singular coordinate transformation, the material parameters of the cloak thus obtained are also singular. This presents serious challenges to practical implementation using metamaterials, as well as making theoretical analysis more difficult. To tackle the acoustic and electromagnetic wave equations with singular coefficients arising from singular cloaking constructions, finite energy solutions in Sobolev spaces with singular weights were introduced and studied in \cite{FullWave, Wormhole,Hetmaniuk,Ting2}. \\
 
A natural way to avoid the singular structures of the perfect cloak is to construct regular approximations which provides cloaking of arbitrary accuracy. Kohn, Shen, Vogelius and Weinstein in \cite{KSVW} developed a regularization method based on approximating the `blow-up-a-point' transformation by a `blow-up-a-small-region' transformation, apart from extending the result of \cite{GKLU0,GKLU1} to $N=2$ dimensions. A slightly different approximation scheme based on truncation of singularities was considered by Greenleaf, Kurylev, Lassas and Uhlmann in \cite{GKLU20}. Nevertheless, it was pointed out that these constructions are equivalent in \cite{KLS}. Due to its practical importance, such approximate cloaking schemes have been studies in various settings, such as acoustic waves \cite{KOVW, Ammari1, Ammari2, FSH, FSH2, LiuHelmholtz, LiuScatter,LiuAcoustic,Nguyen1,Nguyen2},  elastic waves \cite{Hongyu, Lin} and electromagnetic waves \cite{NearCloak,NearCloak2,Zhou,Ammari3}. Also see \cite{LHUG} for a broad survey of approximate cloaking. In \cite{DengLiuUhlmann}, the authors generalize this approximate cloaking construction to more general generating sets, such as curves and planar surfaces.\\

Though the regularized approximate cloak is nonsingular, its material parameters are still anisotropic, which poses another difficulty in practical implementation. Therefore, it would be useful if, as a second step these regularized approximate cloaks were further approximated by regular \emph{isotropic} cloaks.  It is a well-known phenomenon in homogenization theory \cite{A,T,BLP} that homogenization of isotropic material parameters can lead, in the small scale limit, to anisotropic parameters.This approach, called \emph{isotropic transformation optics} was introduced and performed in the case of electrostatic, acoustic and quantum cloaking by Greenleaf, Kurylev, Lassas and Uhlmann in \cite{GKLU2} and \cite{IsoTechnical}. It was shown in \cite{IsoTechnical} that the for a fixed Dirichlet boundary value, the Neumann data for the approximate cloaking construction converge strongly. Moreover, Faraco, Kurylev and Ruiz \cite{Faraco} showed that the same construction in fact gives us strong convergence of the Dirichlet-to-Neumann maps in the operator norm.  More recently, it has been implemented in a quasilinear model \cite{TK}.\\

In this paper, we will construct an approximate isotropic cloak for Maxwell's equations using similar homogenization techniques.  We begin by summarizing the main results of the existence-uniqueness and cloaking theory for the Maxwell equations in Section 2. In Section 3, we present a brief overview of homogenization and present a proof of the homogenization result for the Maxwell equations, under the assumption that both $\epsilon(x)$ and $\mu(x)$ have positive imaginary parts. Section 4 will be devoted to an explicit construction of an approximate isotropic cloak using inverse homogenization techniques as in \cite{GKLU2}. Since the magnetic permeability of the anisotropic approximate cloak does not have an imaginary part, we need to go through a two step process. In the first step, we alter permittivities and permeabilities by a small positive parameter $\delta$ so that the assumptions of Section 3 are satisfied, and then construct isotropic parameters $\eps^n_\delta$ and $\mu^n$ that homogenize to the altered paremetrs. In the second step, we let $\delta \to 0$. Finally in Section 5, we will prove our Main Theorem: that as we first let $n \to \infty$ and then $\delta \to 0$, the electromagnetic boundary measurements corresponding to our approximate cloaking construction converge strongly to the boundary measurements in empty space. The order in which these limits are taken can not be interchanged.

\tableofcontents

\section{Introduction to the Maxwell equations}

Let us begin by introducing the Maxwell system of equations and some basic results regarding existence, uniqueness and stability of solutions. Let $\Omega \subset \R^3$ be a smooth bounded domain with (possibly anisotropic) electric permittivity, magnetic permeability and conductivity described by real symmetric matrix-valued measurable functions $\widetilde{\eps},\mu$ and $\sigma$ respectively. If $J \in L^2(\Omega)^3$ is the intrinsic current density in $\Omega$, the time harmonic Maxwell equations  in $\Omega$ at frequency $\omega >0$ are given by
\begin{equation}\label{maxwell}
\begin{cases} 
\curl E = i\omega\mu H,\\
\curl H = -i\omega\lb \widetilde{\eps} +\frac{i\sigma}{\omega}\rb E + J
\end{cases}
\end{equation}
where $E$ and $H$ are the electric and magnetic fields respectively. For convenience, we define $\eps = \widetilde{\eps} +i\sigma/\omega$. We say that a medium is \emph{regular} if $\widetilde{\eps}$ and $\mu$ are uniformly positive definite and $\sigma$ is non-negative definite over $\Omega$, that is, there exists $C > 0$ such that
\begin{eqnarray}
C|\xi|^2 \leq \sum_{i,j=1}^3 \widetilde{\eps}_{ij}(x)\xi_i\xi_j \leq C^{-1}|\xi|^2, \label{regbig}\\
C|\xi|^2 \leq \sum_{i,j=1}^3 \mu_{ij}(x)\xi_i\xi_j \leq C^{-1}|\xi|^2,
\end{eqnarray}
and 
\begin{equation} \label{regend}
0 \leq \sum_{i,j=1}^3 \sigma_{ij}(x)\xi_i\xi_j \leq C^{-1}|\xi|^2 
\end{equation}
for all $\xi \in \R^3$ and a.e. $x \in \Omega$. Under these regularity assumptions on $\eps$ and $\mu$, it can be shown that the boundary value problem for the Maxwell equations is uniquely solvable in a certain Sobolev-type space, except for a discrete set of values of $\omega$. In this section, we introduce the Maxwell system of equations and related function spaces. We will then note (without proof) some of the important results about these spaces that will be used throughout the paper. The material presented here is classical. The reader can refer to \cite{Kirsch, Nedelec, Monk, BUF1, BUF2, GR} for details.

\subsection{Function spaces for the Maxwell equations} 
\begin{defn} Let $\Omega \subset \R^3$ be open. We define the vector space
\[
H(curl,\Omega) = \{u = (u^1,u^2,u^3)\in (L^2(\Omega))^3: \, \curl u \in L^2(\Omega)^3\} \]
and equip it with the inner product
\[
\langle u,w \rangle_{H(curl,\Omega)} = \int_\Omega u\cdot \overline{w} +(\curl u)\cdot\overline{(\curl w)}. \]
Similarly, we define 
\[
H(div,\Omega) = \{u = (u^1,u^2,u^3)\in (L^2(\Omega))^3: \, \nabla \cdot u \in L^2(\Omega)\} \]
with the inner product
\[
\langle u,w \rangle_{H(div,\Omega)} = \int_\Omega u\cdot \overline{w} +(\nabla \cdot u)\overline{(\nabla \cdot w)}. \]
As in the case of the usual Sobolev spaces, $H_0(curl,\Omega)$ and $H_0(div,\Omega)$ are defined to be the closures of $C_c^\infty(\Omega)$ in $H(curl,\Omega)$ and $H(div,\Omega)$ respectively.
\end{defn}
It is easy to show that $H(curl,\Omega)$ and $H(div,\Omega)$ are both Hilbert spaces. Moreover, when $\Omega$ is Lipschitz,  $C^\infty(\overline{\Omega})^3$ is dense in both of these spaces, and there is a canonical way of taking traces of these functions on the boundary $\pa \Omega$.

\begin{defn}
Let $\Omega\subset\mathbb{R}^3$ be an open set with Lipschitz boundary. Let $\nu$ be the outward unit normal vector field on $\pa \Omega$. We define for  $u\in C^\infty(\overline{\Omega})^3$,
\[\begin{cases}
\gamma_n\, u := \nu\cdot u \quad\mbox{on }\partial\Omega\\
\gamma_T\, u := u-(\nu\cdot u)\nu=-\nu\times(\nu\times u)\quad\mbox{on }\partial\Omega\\ 
\gamma_\tau\, u := \nu\times u \quad\mbox{on }\partial\Omega.
\end{cases}
\]

On the boundary of $\Omega$, we define the following Sobolev spaces of tangential fields: For all  $s\in\mathbb{R}$,
\begin{eqnarray*}
TH^{s}(\partial\Omega) &:=& \{u\in (H^{s}(\pa \Omega))^3:\ \gamma_n\, u = 0\mbox{ on }\pa\Omega\}\\
H^{s}(div,\pa \Omega) &:=& \{ u \in TH^{s}(\partial\Omega):\, Div_{\pa \Omega}u \in H^{s}(\pa \Omega) \} \\
H^{s}(curl, \pa \Omega) &:=& \{ u \in TH^{s}(\partial\Omega):\, Curl_{\pa \Omega}u \in H^{s}(\pa \Omega) \}.
\end{eqnarray*}
Here $Div_{\pa \Omega}$ and $Curl_{\pa \Omega}$ represent the surface divergence and surface curl respectively, defined by
\begin{equation*}
\begin{cases}
\langle Div_{\pa \Omega}u,\varphi\rangle &= -\langle u,\nabla_T \varphi\rangle\\
\langle Curl_{\pa \Omega}u,\varphi\rangle &= -\langle u,\nu\times \nabla_T \varphi\rangle 
\end{cases}
\qquad \forall \varphi \in C^{\infty}(\pa \Omega)
\end{equation*}
where $\nabla_T$-denotes the tangential component of the gradient and $\langle,\rangle$ denotes the standard $L^2(\pa\Omega)$ inner product with respect to the surface measure on the boundary $\pa\Omega$. \\ 
\end{defn}
\noindent
Note that, the surface curl $Curl_{\pa\Omega}u$  is a scalar function and enjoys the relation 
\[Curl_{\pa \Omega}u= Div_{\pa \Omega}(u\times \nu),\quad  u\in H^{s}(curl, \pa \Omega).\]
In particular, when $u$ is smooth,
\[
Curl_{\pa \Omega}u = \nu \cdot (\curl u)|_{\pa \Omega}. \]

\begin{theorem}[Trace Theorem]\label{t3} There exist continuous  extensions of the trace operators $\gamma_{\tau}: u \mapsto \nu \times u|_{\pa \Omega}$ and $\gamma_T: u \mapsto -\nu\times(\nu\times u|_{\pa \Omega})$ to continuous linear operators from $H(curl,\Omega) \to H^{-1/2}(div,\pa \Omega)$ and $H(curl,\Omega) \to H^{-1/2}(curl,\pa \Omega)$ respectively. Moreover, the normal trace operator $\gamma_n : u \mapsto \nu\cdot u|_{\pa\Omega}$ extends  to a continuous linear operator from $H(div,\Omega) \to H^{-1/2}(\Omega)$. The kernels of these operators are precisely the closures of $C_c^\infty(\Omega)^3$ in the corresponding domain spaces, that is,
\[
\ker \gamma_{\tau} = \ker \gamma_T = H_0(curl,\Omega), \qquad \ker \gamma_n =H_0(div,\Omega). \]

Moreover, $\gamma_{\tau},\gamma_T$ and $\gamma_n$ are surjective and have bounded right inverses $\eta_{\tau}: H^{-1/2}(div,\pa \Omega) \to H(curl,\Omega)$, $\eta_T: H^{-1/2}(curl,\Omega) \to H(curl,\Omega)$ and $\eta_n: H^{-1/2}(\pa \Omega)\to H(div,\Omega)$, respectively.
\end{theorem}
\begin{lemma}[Duality]\label{t2}
The spaces $H^{-1/2}(div,\pa\Omega)$ and $H^{-1/2}(curl,\pa\Omega)$ are mutually adjoint
with respect to the scalar product in $TL^2(\pa\Omega)$.
Moreover, for any $u_1,u_2\in  H(curl,\Omega)$ , the following Stokes formula holds
\[\int_\Omega \left( u_1\cdot (\nabla\times u_2)- u_2\cdot (\nabla\times u_1)\right)\, dx = \langle \gamma_T\, u_2 , \gamma_\tau u_1\rangle\]
where $\langle,\rangle$ denotes the duality bracket between $H^{-1/2}(curl,\pa\Omega)$ and $H^{-1/2}(div,\pa\Omega)$ 
with respect to the $L^2(\pa\Omega)$ inner product.
\end{lemma}

Notice that if $u \in H_0(curl,\Omega)$, it does not imply that $u$ vanishes on the boundary $\pa \Omega$, but only that the \textit{tangential component} $\nu \times u \equiv 0$ on $\pa \Omega$. For example, let $B_1(0)$ represent the unit ball centered at the origin in $\R^3$ and consider the radial vector field $u(x)=x$ in $H(curl,B_1(0))$. Clearly, $u$ does not vanish on the boundary but its tangential component does, and therefore $u \in H_0(curl,B_1(0))$ by the above theorem. Similarly, if $u \in H_0(div,\Omega)$, we can only conclude that the normal component of $u$ vanishes on $\pa \Omega$. \\

An important property of these Sobolev spaces is the Div-Curl lemma, which is a crucial ingredient of the proof the homogenization result Theorem \ref{Homogenization}, as well as the proof of our Main Theorem in Section 5.  A proof can be found in \cite[Theorem 11.2]{BLP}.

\begin{theorem}[The Div-Curl Lemma]
Suppose 
\begin{eqnarray*}
u^n & \rightharpoonup & u \quad \mbox{in } H(curl,\Omega), \, \mbox{ and }\\
w^n & \rightharpoonup & w \quad \mbox{in } H(div,\Omega).
\end{eqnarray*}
Then
\[
u^n \cdot w^n \rightarrow u\cdot w \quad \mbox{in }\mathcal{D}'(\Omega), \]
that is,
\[
\int_\Omega u^n \cdot w^n \varphi \to \int_\Omega u \cdot w \varphi \qquad \forall \varphi \in C_c^\infty(\Omega). \]
\end{theorem}

We can now state the existence and uniqueness theorem for boundary value problems for the time harmonic Maxwell equations.

\begin{theorem}\label{existunique} Let $\widetilde{\eps},\mu,\sigma$ satisfy the conditions \eqref{regbig}-\eqref{regend} and let $J \in L^2(\Omega)^3$. Then there exists a discrete subset $F \subset \R$ such that for all $\omega \in \R \setminus F$, the following boundary value problem:
\begin{equation}\label{origmaxwell}
\begin{cases}
\nabla \times E = i\omega \mu H\\
\nabla \times H = -i\omega\lb \widetilde{\eps}+\frac{i\sigma}{\omega}\rb E + J\\
\nu \times E|_{\pa \Omega} = f \in H^{-1/2}(div,\pa \Omega)
\end{cases}
\end{equation}
has a unique solution $(E,H) \in H(curl,\Omega)\times H(curl,\Omega)$. Moreover, we have the estimate
\begin{equation*}
\|E\|_{H(curl, \Omega)} +\|H\|_{H(curl,\Omega)} \leq C(\omega)(\|f\|_{H^{-1/2}(div,\pa \Omega)}+\|J\|_{L^2(\Omega)}). 
\end{equation*}
On the other hand, if $\omega \in F$, there exist non-zero solutions to the corresponding homogeneous system
\begin{equation*}
\begin{cases}
\nabla \times E = i\omega \mu H\\
\nabla \times H = -i\omega\lb \widetilde{\eps}+\frac{i\sigma}{\omega}\rb E \\
\nu \times E|_{\pa \Omega} = 0.
\end{cases}
\end{equation*}
\end{theorem}
We call $F$ the set of electromagnetic eigenvalues for $(\Omega;\eps,\mu)$.

\subsection{Cloaking for the Maxwell system}

In the rest of this paper, we assume that $J=0$, that is, there are no internal sources of current in the medium. Suppose $\omega \in \R $ is not an eigenvalue for $(\Omega;\eps,\mu)$. Define $\Lambda_{\eps,\mu}: H^{-1/2}(\mbox{div},\pa \Omega) \to H^{-1/2}(\mbox{div},\pa \Omega)$  by
\begin{equation*}
\Lambda_{\eps,\mu}(f) = \nu \times H|_{\pa \Omega},
\end{equation*}
where $(E,H)$ is the unique solution of \eqref{origmaxwell}. In other words, $\Lambda_{\eps,\mu}$ (also called the \emph{impedance map}) maps the tangential component of the electric field on the boundary to the tangential component of the magnetic field on the bounday. It follows from Theorem \ref{existunique} that $\Lambda_{\eps,\mu}$ is a continuous linear map. It is the analogue of the Dirichlet-to-Nuemann map for the Maxwell equations and encodes all the possible electromagnetic measurements that can be made at the boundary of $\Omega$.  

\begin{defn}\label{arb} Let $D \Subset \Omega$ be a smooth bounded subdomain and let $\widetilde{\eps_c},\mu_c$ and $\sigma_c$ be real symmetric matrix valued measurable functions on $\Omega \setminus \overline{D}$. Define $\eps_c = \widetilde{\eps_c}+i\sigma_c/\omega$ as before. We say that $(\Omega\setminus \overline{D};\eps_c,\mu_c)$ is a \textbf{cloak} for the region $D$, if for any \emph{regular} permittivity, permeability and conductivity $\widetilde{\eps_a},\mu_a,\sigma_a$ defined in $\Omega$, we have
\begin{equation*}
\Lambda_{\eps_e,\mu_e} = \Lambda_{I,I} \qquad \mbox{on } \pa \Omega,
\end{equation*}
where $I$ is the $3 \times 3$ identity matrix and $\eps_e,\mu_e$ are the electromagnetic parameters of the extended object:
\begin{equation*}
(\eps_e,\mu_e) = \begin{cases}
                          (\widetilde{\eps_c}+i\sigma_c/\omega,\mu_c) & \mbox{in } \Omega \setminus \overline{D}\\
                          (\widetilde{\eps_a}+i\sigma_a/\omega,\mu_a) & \mbox{in } D
                          \end{cases}
\end{equation*}
\end{defn}
In other words, the boundary measurements $\Lambda_{\eps_e,\mu_e}$ must be indistinguishable from the measurements one would get if $\Omega$ was a vacuum (which corresponds to $\eps=\mu=I$). We call $\Omega \setminus \overline{D}$ and $D$ the cloaking region and the cloaked region respectively. \\

The following invariance property of the Maxwell equations forms the basis of the transformation optics method of constructing electromagnetic cloaks: Let $\widetilde{\Omega}$ be another smooth bounded domain and suppose there exists a smooth diffeomorphism $F: \Omega \to \widetilde{\Omega}$.
\begin{lemma}\label{chcor} Suppose $(E,H) \in H(curl,\Omega)\times H(curl,\Omega)$ satisfies the Maxwell equations (\ref{maxwell}) in $\Omega$ with $J=0$. Define the EM fields
\begin{equation*}
\widetilde{E} = F_*E := (DF^t)^{-1}E\circ F^{-1}, \qquad \widetilde{H} = F_*H = (DF^t)^{-1}H\circ F^{-1}
\end{equation*}
on $\widetilde{\Omega}$. Then $(\widetilde{E},\widetilde{H}) \in H(curl,\widetilde{\Omega})\times H(curl,\widetilde{\Omega})$ and satisfy
\begin{equation}\label{maxwell2}
\begin{cases} 
\curl \widetilde{E} = i\omega\mu'\widetilde{H},\\
\curl \widetilde{H} = -i\omega\lb \widetilde{\eps}' +\frac{i\sigma'}{\omega}\rb \widetilde{E}
\end{cases}
\end{equation}
where $\widetilde{\eps}',\mu',\sigma'$ are the \emph{push-forwards} of $\widetilde{\eps},\mu,\sigma$ under $F$, defined by
\begin{eqnarray*}
\widetilde{\eps}' &=& F_* \widetilde{\eps} := \lb \frac{1}{\det(DF)}(DF) \cdot \widetilde{\eps}\cdot (DF)^t\rb \circ F^{-1},\\
\mu' &=& F_* \mu := \lb \frac{1}{\det(DF)}(DF) \cdot \mu\cdot (DF)^t\rb \circ F^{-1},\\
\sigma' &=& F_* \sigma := \lb \frac{1}{\det(DF)}(DF) \cdot \sigma\cdot (DF)^t\rb \circ F^{-1}.
\end{eqnarray*}
\end{lemma}
A proof can be found in \cite{Zhou}. As an immediate consequence, we have the following corollary.
\begin{corollary}\label{chcorDN} Let $(\Omega;\eps,\mu)$ be as before and suppose $F:\overline{\Omega} \to \overline{\Omega}$ is a diffeomorphism such that $F|_{\pa \Omega} = Id$. Then
\begin{equation*}
\Lambda_{\eps,\mu} = \Lambda_{F_*\eps,F_*\mu} \qquad \mbox{on }\pa \Omega.
\end{equation*}
\end{corollary}
\begin{proof}
Fix $f \in H^{-1/2}(div,\pa \Omega)$. Let $(E,H)$ and $(\widetilde{E},\widetilde{H})$ be the solutions of (\ref{maxwell}) (with $J=0$) and (\ref{maxwell2}) respectively, with the boundary condition
$$ \nu \times E = \nu \times \widetilde{E} = f.$$
Note that $F|_{\pa \Omega} = Id$ and consequently, $(DF_x)|_{T_x \pa \Omega} =Id$ for all $x \in \pa \Omega$. Therefore
\begin{equation*}
\Lambda_{F_*\eps,F_*\mu}f = \nu \times \widetilde{H}|_{\pa \Omega} = (DF^t)^{-1}(\nu \times H)\circ F^{-1}|_{\pa \Omega} =\nu \times H|_{\pa \Omega} =\Lambda_{\eps,\mu}f.
\end{equation*}
\end{proof}

\begin{remark} In Lemma \ref{chcor} and Corollary \ref{chcorDN}, it is sufficient to assume that $\Omega, \widetilde{\Omega}$ are Lipschitz and that $F$ is bi-Lipschitz.
\end{remark}

We now give a brief account of the singular ideal cloak and the regular approximate cloak constructed in \cite{FullWave} and \cite{NearCloak} respectively. Henceforth, let $B_r$ denote $\{x\in \R^3:|x|<r\}$. Consider the map $F: \overline{B}_2 \setminus \{0\} \to \overline{B}_2 \setminus \overline{B}_1$ given by
\begin{equation*}
F(x) = \lb 1+\frac{|x|}{2}\rb \frac{x}{|x|},
\end{equation*}
which blows up the Origin to the cloaked region $\overline{B}_1$. The central idea of transformation optics-based electromagnetic cloaking is to use the invariance properties \ref{chcor} and \ref{chcorDN} to ``hide" an arbitrary regular electromagnetic object in $B_1$. Notice however, that $F$ is not a regular change of coordinates, as $\det (DF) \to 0$ as $|x|\to 0$. Define the electromagnetic parameters in $\overline{B}_2\setminus \overline{B}_1$ to be the pushforwards of the identity under the coordinate transformation $F$:
\begin{equation*}
\widetilde{\eps}(x) = \widetilde{\mu}(x) = F_*I = \frac{(DF)\cdot I \cdot (DF)^t}{|\det DF|}(F^{-1}(x)), \quad 1<|x|\leq 2. 
\end{equation*}
More explictly,
\begin{equation*}
\widetilde{\eps}=\widetilde{\mu} = 2\frac{(|x|-1)^2}{|x|^2}\Pi(x) +2(I-\Pi(x))
\end{equation*}
where $\Pi(x)$ is the projection map in the radial direction.
\begin{equation*}
\Pi(x)w = \lb w\cdot \frac{x}{|x|}\rb \frac{x}{|x|}.
\end{equation*}
It is easy to see that one of the eigenvalues of $\eps$ (and $\mu$), namely, the one corresponding to the radial direction goes to $0$ as $|x| \to 1$. Therefore, the regularity conditions \eqref{regbig}-\eqref{regend} no longer hold. Again consider the extended object
\begin{equation*}
(B_2;\widetilde{\eps}_e,\widetilde{\mu}_e) := 
\begin{cases}
(B_2\setminus \overline{B}_1;\widetilde{\eps},\widetilde{\mu}) & \mbox{in }B_2 \setminus \overline{B}_1\\
(B_1;\eps_a,\mu_a) & \mbox{in }B_1
\end{cases}
\end{equation*}
where $\eps_a$ and $\mu_a$ are arbitrary. It was shown in \cite{FullWave} that distributional solutions $(\widetilde{E},\widetilde{H})$ of the Maxwell equations corresponding to $\widetilde{\eps}_e$ and $\widetilde{\mu}_e$ satisfying certain ``finite energy" conditions have the same Cauchy data as the solutions of the Maxwell equations in free space. More precisely, if $(E,H) \in H(curl,\Omega)\times H(curl,\Omega)$ is the solution of
\begin{equation}\label{freespace}
\begin{cases}
\curl E = i\omega H, \\
\curl H = -i\omega E,\\
\nu \times E|_{\pa \Omega} = \nu \times \widetilde{E}|_{\pa \Omega}
\end{cases}
\end{equation}
then
$$ \nu \times \widetilde{H}|_{\pa\Omega} = \nu \times H|_{\pa \Omega}.$$
Thus, $(B_2\setminus \overline{B}_1;\widetilde{\eps},\widetilde{\mu})$ forms a perfect cloak for the region $B_1$. However, the singular nature of the parameters of this cloak poses serious challenges to practical implementation as well as theoretical analysis. Therefore, it is natural to consider regular approximations to the perfect cloak. As a trade-off, we lose the perfect cloaking of the singular construction and instead obtain an \emph{approximate} cloak.

\begin{defn}\label{appr} Let $D \Subset \Omega$ be a smooth bounded subdomain and let $\widetilde{\eps_c}^\rho,\mu_c^\rho$ and $\sigma_c^\rho$ be real symmetric matrix valued measurable functions on $\Omega \setminus \overline{D}$, indexed by a positive real number $\rho$. Define $\eps_c^\rho = \widetilde{\eps_c}^\rho+i\sigma_c^\rho/\omega$ as before. We say that $(\Omega\setminus \overline{D};\eps_c^\rho,\mu_c^\rho)$ is an \textbf{approximate cloak} for the region $D$, if for any \emph{regular} permittivity, permeability and conductivity $\widetilde{\eps_a},\mu_a,\sigma_a$ defined in $\Omega$, we have
\begin{equation*}
\|\Lambda_{\eps_e^\rho,\mu_e^\rho} - \Lambda_{I,I}\|_{\mathcal{L}(H^{-1/2}(div,\pa \Omega), H^{-1/2}(div,\pa \Omega))} \to 0 \qquad \mbox{as } \rho \to 0+.
\end{equation*}
where $I$ is the $3 \times 3$ identity matrix and, as before $\eps_e^\rho,\mu_e^\rho$ are the electromagnetic parameters of the extended object:
\begin{equation*}
(\eps_e^\rho,\mu_e^\rho) = \begin{cases}
                          (\widetilde{\eps_c^\rho}+i\sigma_c^\rho/\omega,\mu_c) & \mbox{in } \Omega \setminus \overline{D}\\
                          (\widetilde{\eps_a}+i\sigma_a/\omega,\mu_a) & \mbox{in } D
                          \end{cases}
\end{equation*}
Here, $\|\cdot\|_{\mathcal{L}(H^{-1/2}(div,\pa \Omega),H^{-1/2}(div,\pa \Omega))}$ denotes the standard operator norm on bounded linear maps from $H^{-1/2}(div,\pa \Omega)$ to itself.
\end{defn}

Two different approximation schemes have been proposed for approximate cloaking. In \cite{GKLU20}, Greenleaf et. al. proposed that one use the transformation $F$ to blow up a small ball $B_\rho$ to a larger ball $B_R$, where $1<R<2$, whereas Kohn et. al. in \cite{KSVW} use a transformation similar to $F$ to blow up $B_\rho$ to $B_1$. In the present paper, we work with the latter approximation scheme. However, it is known that the former approximation scheme gives a similar performance \cite{Zhou}. \\

Now, for $0 < \rho < 1$, consider the bi-Lipschitz transformation $F_\rho: \overline{B}_3 \to \overline{B}_3$ given by
\begin{equation}\label{frho}
F_\rho (x) = \begin{cases}
x & \mbox{for } 2\leq |x| \leq 3,\\
\lb \frac{2(1-\rho)}{2-\rho}+\frac{|x|}{2-\rho}\rb \frac{x}{|x|} & \mbox{for }\rho < |x| \leq 2,\\
\frac{x}{\rho} & \mbox{for }|x|\leq \rho,
\end{cases}
\end{equation}
We see that $F_\rho$ dilates $B_\rho$ to $B_1$ and retracts $\overline{B}_3 \setminus B_\rho$ to $\overline{B}_3 \setminus B_1$. Also note that we are working in a slightly larger ball $B_3$ rather than $B_2$ and that $F_\rho = Id$ in the annulus $\overline{B}_3 \setminus B_2$. The advantage of this modification is that the pushed-forward electric permittivity and magnetic permeability will both be identity in a neighbourhood of the boundary of our domain, which will prove useful in the proof of our Main Theorem in Section 5. 

\begin{theorem}[\cite{NearCloak}]\label{Nearcloak} Given $0<\rho<1$, consider the electromagnetic object $(B_3,\eps_e^\rho,\mu_e^\rho)$ with electromagnetic parameters defined by
\begin{equation}\label{rho}
(\eps_e^\rho,\mu_e^\rho) = \begin{cases}
 ((F_\rho)_*I, (F_\rho)_*I) & \mbox{in } B_3\setminus B_1,\\
( (F_\rho)_*(1+i\rho^{-2}/\omega)I, (F_\rho)_*I)& \mbox{in } B_1 \setminus B_{1/2},\\
(\eps_a,\mu_a)&  \mbox{in } B_{1/2}
\end{cases}
\end{equation}
where $\eps_a, \mu_a$ are arbitrary and regular. Suppose also that $\omega \in \R$ is not an EM eigenvalue of $(B_3;I,I)$. Then there exists $\rho_0 > 0$ such that for all $0 < \rho <\rho_0$, $\omega$ is not an eigenvalue for $(B_3,\eps_e^\rho,\mu_e^\rho)$ for any choice of regular $\eps_a,\mu_a$. Moreover,
\[
\|\Lambda_{\eps_e^\rho,\mu_e^\rho}-\Lambda_{I,I}\|_{\mathcal{L}(H^{-1/2}(div,\pa B_3), H^{-1/2}(div,\pa B_3))} \to 0 \quad \mbox{as } \rho \to 0+. \]
\end{theorem}

\begin{remark}
The above theorem has been stated in greater generality in \cite{NearCloak}, with $B_3$ and $B_1$ replaced by arbitrary bounded simply connected smooth domains. But the specific case we have considered is sufficient for our purposes.
\end{remark}

Notice that apart from regularizing $F$, we have added a layer of high conductivity in the region $B_1 \setminus B_{1/2}$. This additional layer was shown to be essential in \cite{Zhou}, since in the absence of this layer, one can have $\eps_a,\mu_a$ that make $\omega$ an eigenvalue.

\section{Homogenization and the Maxwell equations}

In this section, we recall the notion of $H$-convergence and present a homogenization result for the Maxwell equations.
\begin{defn} Let $0 < \alpha < \beta < \infty$ and let $\Omega \subset \R^N$ be open. We define $\mathcal{M}_{\mathbb{R}}(\alpha, \beta;\Omega)$ to be the set of all real $N\times N$ matrix-valued measurable functions $A(x)=[a_{kl}(x)]_{1\leq k,l\leq N}$ defined almost everywhere on $\Omega$  such that 
\begin{eqnarray*} (A(x)\xi,\xi)&=& \sum_{k,l=1}^Na_{kl}(x)\xi_k\xi_l \geq \alpha|\xi|^2, \quad \mbox{and} \\
|A(x)\xi| &\leq & \beta|\xi|
\end{eqnarray*}
for all $\xi \in \mathbb{R}^N$ and  a.e. $x\in\Omega$.\\

Analogously, in the complex case, we define $\mathcal{M}_{\mathbb{C}}(\alpha, \beta;\Omega)$ to be the set of all complex $N\times N$ matrix-valued measurable functions  $A(x)$ defined almost everywhere on $\Omega$ such that
\begin{eqnarray*}
-\mathrm{i}\xi\cdot\lb A(x)-A(x)^{\dagger}\rb\overline{\xi} &\geq & \alpha|\xi|^2, \quad \mbox{and}\\
\ |A(x)\xi| &\leq & \beta|\xi|
\end{eqnarray*}
for all $\xi \in \mathbb{C}^N$ and a.e. $x\in\Omega$. Here $A(x)^{\dagger} = (\overline{A(x)})^t = \overline{A(x)^t}$ denotes the Hermitian conjugate of $A(x)$. Note that $A(x) \in \mathcal{M}_\C(\alpha,\beta;\Omega)$ implies that $\Re(A(x))$ (real part of the matrix $A(x)$) is symmetric, i.e. $\Re(A(x))_{kl}=\Re(A(x))_{lk}$ for all $1\leq k,l\leq N$.

We will say $A\in \mathcal{M}(\alpha,\beta;\Omega)$ if either $A\in\mathcal{M}_{\mathbb{R}}(\alpha,\beta;\Omega)$ or $A\in\mathcal{M}_{\mathbb{C}}(\alpha,\beta;\Omega)$.
\end{defn}

Next, we define the notion of $H$-convergence \cite{A,T}:
\begin{defn}
Let $A^n\in \mathcal{M}(\alpha,\beta; \Omega)$ for $ n \in \N$ and $A^{*}\in \mathcal{M}(\widetilde{\alpha},\widetilde{\beta}; \Omega)$. We say 
\[\mbox{$A^n \xrightarrow{H} A^{*}$ or $H$-converges to 
 $A^{*}$}\]
if for all test sequences $u^n \in H^1(\Omega)$ satisfying 
\begin{align*}
u^{n} &\rightharpoonup u \quad\mbox{weakly in }H^1(\Omega)\\
-\nabla \cdot(A^n\mathbb\nabla u^n)& \mbox{ is strongly convergent in } H^{-1}(\Omega).
\end{align*}
we have $A^n\nabla u^n \rightharpoonup A^{*}\nabla u$ in $L^2(\Omega)^N$. We call $A^*$ the homogenized matrix for the sequence $\{A^n\}$.
\end{defn}

To illustrate the utility of  $H$-convergence, let us consider the following sequence of elliptic boundary value problems: Let $A^n$ be a sequence of matrices in $\mathcal{M}(\alpha,\beta;\Omega)$ such that $A^n \xrightarrow{H} A^*$ and let $u_n$ be the solutions of
\begin{equation*}
\begin{cases}
-\nabla \cdot (A^n\nabla u_n) = f \in H^{-1}(\Omega),\\
u_n|_{\pa \Omega} = 0.
\end{cases}
\end{equation*}
It is easy to see that $\|u_n\|_{H^1(\Omega)} \leq C$ for some constant $C$ independent of $n$. Therefore, there exists a subsequence, which we still denote by $u_n$ that converges weakly to some $u \in H_0^1(\Omega)$. Now, the definition of $H$-convergence implies that $A^n\nabla u_n \rightharpoonup A^*\nabla u$ as $n \to \infty$. Consequently, $0 = -\nabla \cdot (A^n \nabla u_n) \rightharpoonup -\nabla \cdot (A^*\nabla u)$. Therefore, the weak limit $u$ of $u_n$ is in fact the solution of the following ``homogenized" problem:
\begin{equation*}
\begin{cases}
-\nabla \cdot (A^*\nabla u) = f \quad \mbox{in }\Omega,\\
u|_{\pa \Omega} = 0.
\end{cases}
\end{equation*}

\subsection{Homogenization with periodic micro-structures}
For some types of sequences $A^n$, the homogenized matrix $A^*$ can be explicitly computed. Let us consider the class of periodic micro-structures as an example. Let $Y$ denote the unit cube $[0,1]^N$ in $\mathbb{R}^N$.
Let $A(y)=[a_{kl}(y)]_{1\leq k,l\leq N}\in \mathcal{M}(\alpha,\beta; Y)$  be such that $a_{kl}(y)$ are $Y$-periodic functions
$ \forall k,l =1,2..,N$, that is, $a_{kl}(y + z) = a_{kl}(y)$ whenever $z \in \mathbb{Z}^N$ and $y \in Y$. 
Now we set 
\[
A^{n}(x) = [a_{kl}^{n}(x)]= [a_{kl}(nx)], \qquad x \in [0,1/n]^N
\] and extend 
it to all of $\mathbb{R}^N$ by $1/n$-periodicity. The restriction of $A^{n}$ to $\Omega$ is known as a periodic micro-structure. Such micro-structures arise in the study of physical systems where the parameters vary periodically, with a period that is very small compared to the dimensions of the object under consideration.\\

In this classical case, the homogenized conductivity $A^{*}=[a^{*}_{kl}]$ is a constant matrix whose entries are given by \cite{A,BLP}
\begin{equation*}
a^{*}_{kl} = \int_{Y}\sum_{i,j=1}^Na_{ij}(y)\frac{\partial}{\partial y_i}(\chi_k(y) + y_k)\frac{\partial}{\partial y_j}(\chi_l(y) + y_l)dy
\end{equation*}
where we define the $\chi_k$ through the so-called cell-problems. For each canonical basis vector $e_k$,  $\chi_k$ is defined to be the unique solution of the conductivity problem in the periodic unit cell :
\begin{equation*}
-\nabla_y \cdot (\ A(y)(\nabla_y\chi_k(y)+e_k)) = 0 \quad\mbox{in }\mathbb{R}^N,\quad y \mapsto \chi_k(y) \quad\mbox{is $Y$-periodic }
\end{equation*}
in the Sobolev space $H^1_{\#,0}(Y)$ defined by
\[
H^1_{\#,0}(Y) := \left\{ f \in H^1_{loc}(\R^N): \, y \mapsto f(y) \mbox{ is $Y$-periodic}, \, \int_Y f(y)dy =0 \right\}. \]
We can generalize the above case to what are called \textit{locally periodic} micro-structures. Let  $A(x,y)=[a_{kl}(x,y)]_{1\leq k,l\leq N}\in \mathcal{M}(\alpha,\beta; \Omega\times Y)$  be such that $a_{kl}(x,\cdot)$ are $Y$-periodic functions with respect to the second variable $ \forall k,l =1,2,..,N$ and for almost every $x$ in $\Omega$. Now we set 
\[
A^{n}(x) = [a_{kl}^{n}(x)]= [a_{kl}(x,nx)]
\]
Then the homogenized conductivity $A^{*}(x)=[a^{*}_{kl}(x)]$ is given by \cite{BLP,OL}
\begin{equation}\label{localperiod}
a^{*}_{kl}(x) = \int_{Y}\sum_{i,j=1}^Na_{ij}(x,y)\frac{\partial}{\partial y_i}(\chi_k(x,y) + y_k)\frac{\partial}{\partial y_j}(\chi_l(x,y) + y_l)dy
\end{equation}
where  $\chi_k(x,\cdot)\in H^1_{\#,0}(Y)$ solves the following cell problem for almost every $x$ in $\Omega$:
\[
-\nabla_y \cdot (\ A(x,y)(\nabla_y\chi_k(x,y)+e_k)) = 0 \quad\mbox{in }\mathbb{R}^N,\quad y \mapsto \chi_k(x,y) \quad\mbox{is $Y$-periodic. }
\]

Notice that even if each $A^n(x)$ is a scalar matrix, the homogenized matrix $A^*(x)$ need not be scalar valued. This is the crucial property of homogenization that allows us to approximate anisotropic permittivities and permeabilities by isotropic ones. We note two important properties of $H$-convergence:
\begin{proposition}\label{homprop} Suppose $A^n \xrightarrow{H} A^*$ in $\Omega$. Then
\begin{enumerate}
\item $A^n|_{\Omega'} \xrightarrow{H} A^*|_{\Omega'}$ for all open sets $\Omega' \subset \Omega$.
\item For any matrix-valued function $M(x)$, if $MA^n \in \mathcal{M}(\alpha,\beta,\Omega)$, then $MA^n \xrightarrow{H} MA^*$. 
\end{enumerate}
\end{proposition}
The proof can be found in \cite{A}.

\subsection{Homogenization of the Maxwell system}
Now, we will show how $H$-convergence can be used to find the limit of a sequence of Maxwell systems. Let $\{\eps^n\},\{\mu^n\}$ be sequences such that $\epsilon^n, \mu^n \in\mathcal{M}_{\mathbb{C}}(\alpha,\beta;\Omega)$ $\forall n\in\mathbb{N}$ for some $0<\alpha<\beta$. Then consider the following sequence of time-harmonic Maxwell equations at frequency $\omega >0$:
\begin{equation}\label{newmaxwell}\begin{cases}
\nabla \times E^n = i\omega \mu^n H^n\\
\nabla \times H^n = -i\omega \epsilon^n E^n \\
\nu \times E^n|_{\pa \Omega} = f \in  H^{-1/2}(div,\pa \Omega)
\end{cases}
\end{equation}

It can be shown (see Proposition \ref{peu} below) that the above problem admits a unique solution 
$(E^n,H^n)\in H(curl,\Omega)\times H(curl,\Omega)$ which satisfies
\begin{equation*}
\|E^n\|_{H(curl, \Omega)} +\|H^n\|_{H(curl,\Omega)} \leq C\|f\|_{H^{-1/2}(div,\pa \Omega)},
\end{equation*}
where the constant $C=C(\alpha,\beta,\omega)$ is independent of $n$. Therefore, $E^n,H^n$ have subsequences (still denoted by $E^n,H^n$) such that
\[
(E^n,H^n) \rightharpoonup (E,H) \mbox{ weakly in } H(curl,\Omega)\times H(curl,\Omega).
\]
Our goal is to get the limiting equation for $(E,H)\in H(curl,\Omega)\times H(curl,\Omega)$.\\

Let us begin with the an existence-uniqueness result for a Maxwell system of the type \eqref{newmaxwell}:
\begin{proposition}\label{peu}
Let $\epsilon, \mu \in\mathcal{M}_{\mathbb{C}}(\alpha,\beta;\Omega)$ for some $0<\alpha<\beta$. Then the following Maxwell system in $\Omega$ at frequency $\omega>0$  
\begin{equation}\label{sys}
\begin{cases}
\nabla \times E = i\omega \mu H\\
\nabla \times H = -i\omega \epsilon E \\
\nu \times E|_{\pa \Omega} = f\in  H^{-1/2}(div,\pa \Omega)
\end{cases}
\end{equation}
admits a unique solution $(E,H)\in H(curl,\Omega)\times H(curl,\Omega)$ satisfying 
\begin{equation}\label{esti}
\|E\|_{H(curl, \Omega)} +\|H\|_{H(curl,\Omega)} \leq C\|f\|_{H^{-1/2}(div,\pa \Omega)},
\end{equation}
where the constant $C=C(\alpha,\beta,\omega)$.
\end{proposition}
\begin{remark}
Note that Theorem \ref{existunique} gives us existence and uniqueness when $\mu,\Re(\epsilon)\in \mathcal{M}_{\mathbb{R}}(\alpha,\beta;\Omega), \Im(\eps) \geq 0$ and for $\omega\in \mathbb{R}\setminus F$. Here we have instead assumed $\mu,\epsilon$ to be in $\mathcal{M}_{\mathbb{C}}(\alpha,\beta;\Omega)$ and $\omega>0$ is arbitrary.    
\end{remark}
\begin{proof}
 Let us consider the equation satisfied by $H$
 \begin{equation*}
  \nabla\times \lb\frac{1}{i\omega}\eps^{-1}(x)(\nabla\times H(x))\rb + i\omega \mu(x)H(x) = 0 \quad\mbox{in }\Omega.
 \end{equation*}
The weak formulation of the above problem would be
\begin{equation}\label{t1}
-\int_\Omega\lb (\nabla\times \overline{w})\cdot \frac{1}{i\omega}\eps^{-1}(\nabla\times H) + i\omega \overline{w}\cdot \mu H\rb dx  = \int_{\partial\Omega} (\nu \times E)\cdot \overline{w}\, dS,\quad \forall w\in H(curl,\Omega).
\end{equation}
Define the sesquilinear form $a: H(curl,\Omega)\times  H(curl,\Omega) \to \mathbb{C} $  by
\begin{equation*}
a(u,w) :=-\int_\Omega\lb (\nabla\times \overline{w})\cdot \frac{1}{i\omega}\eps^{-1}(\nabla\times u) + i\omega \overline{w}\cdot \mu u\rb dx, \quad \forall u,w\in H(curl,\Omega).
\end{equation*}
We note that, by writing $\overline{w} = \gamma_T\, \overline{w} + (\overline{w}\cdot\nu)\nu$ on $\partial\Omega$, the right hand side of \eqref{t1} becomes
\begin{align}\label{t4}
\int_{\partial\Omega} (\nu \times E)\cdot \overline{w}\, dS &= \int_{\partial\Omega} (\nu \times E)\cdot \gamma_T\, \overline{w}\, dS + \int_{\partial\Omega} (\nu \times E)\cdot  (\overline{w}\cdot\nu)\nu\, dS,\quad w\in H(curl,\Omega)\notag\\
&= \int_{\partial\Omega} (\nu \times E)\cdot \gamma_T\, \overline{w}\, dS \quad\mbox{(since $(\nu\times E)\cdot \nu =0 $ on $\pa\Omega$)}.
\end{align}
Then the weak formulation of the problem can be rewritten as
\begin{equation}\label{lm}
 a(H,w) = \langle f, \gamma_T\, w\rangle,\quad\forall w\in H(curl,\Omega)
\end{equation}
where $\langle , \rangle$ denotes the duality bracket between $H^{-1/2}(div,\partial\Omega)$ and $H^{-1/2}(curl,\partial\Omega)$ (cf. Lemma \ref{t2}) with respect to the $TL^2(\pa\Omega)$ inner product as in \eqref{t4}. \\
 We now claim that $a(\cdot,\cdot)$ is coercive, that is,
\begin{equation}\label{co}
 \Re a(u,u) = -\int_\Omega\lb \frac{1}{i\omega}(\nabla\times \overline{u})\cdot\lb(\eps^{-1})^\dagger-{\eps^{-1}}\rb(\nabla\times u) + i\omega u^{*}\cdot\lb\mu-\mu^{\dagger}\rb u \rb \geq C||u||^2_{H(curl,\Omega)}.
\end{equation}
Indeed, from the definition of $\mathcal{M}_\C(\alpha,\beta;\Omega)$ we have
\begin{equation*}
\begin{cases}
-i\omega\xi\cdot\lb\mu(x)-\mu(x)^{\dagger}\rb\overline{\xi}\geq C_1|\xi|^2\\
-\frac{i}{\omega}\xi\cdot\lb\eps(x)-{\eps}^{\dagger}(x)\rb\overline{\xi}\geq C_2|\xi|^2 
\end{cases}
\forall\xi \in \mathbb{C}^3,\ \mbox{ a.e. }x\in\Omega.
\end{equation*}
Now, for any $\zeta \in \C^3$,  
\[
 -\frac{i}{\omega}(\eps^{t}\zeta)\cdot \lb (\eps^{-1})^\dagger-\eps^{-1}\rb\overline{(\eps^{t}\zeta)} = -\frac{i}{\omega}\zeta\cdot \lb\eps -\eps^\dagger\rb\overline{\zeta}\geq C_2|\zeta|^2. 
\]
Setting $\zeta=(\eps^{t})^{-1}\xi$, we conclude that
\[
-\frac{i}{\omega}\xi \cdot ((\eps^{-1})^\dagger-\eps^{-1})\overline{\xi} \geq C_3|\xi|^2, \qquad \forall \xi \in \C^3.
\]
Therefore, by \eqref{co},
\begin{equation*}
\Re a(u,u) \geq C_3\|\nabla \times u\|_{L^2}^2 +C_1\|u\|_{L^2}^2 \geq \min\{C_1,C_3\}\|u\|_{H(curl,\Omega)}^2,
\end{equation*}
which proves the claim. Moreover, since $|\mu\xi|,|\epsilon\xi|\leq \beta|\xi|$ for all $\xi\in\ \C^3$ we can see that the sesquilinear form $a$ is continuous, i.e.,
\[
 |a(u,w)|\leq C||u||_{H(curl,\Omega)}||w||_{H(curl,\Omega)}.
\]
From the standard trace theory (cf. Lemma \ref{t2} and Theorem \ref{t3}) it is easy to see that the right-hand side of \eqref{lm} is continuous on $H(curl,\Omega)$, i.e.,
\begin{align}\label{f}
|\langle f, \gamma_T\, w\rangle|&\leq C ||f||_{H^{-1/2}(div,\partial\Omega)}||w||_{H^{-1/2}(curl,\partial\Omega)}\notag\\
&\leq C^\prime ||f||_{H^{-1/2}(div,\partial\Omega)}||w||_{H(curl,\Omega)},\quad w\in H(curl,\Omega).
\end{align}
Therefore, \eqref{lm} has a unique solution  $ H\in H(curl,\Omega)$ by the Lax-Milgram theorem, and from \eqref{co},\eqref{f} it enjoys the norm estimate
\[ ||H||_{H(curl,\Omega)}\leq C||f||_{H^{-1/2}(div,\partial\Omega)}.\]
Finally, we set $E = \frac{i}{\omega}\eps^{-1}\curl H$. It is easy to check that the weak curl of $E$ is $i\omega \mu H$, and therefore \eqref{sys} and \eqref{esti} easily follow.
\end{proof}

We now prove the homogenization result for the Maxwell system \eqref{newmaxwell}. Our proof closely follows the  method used in \cite{BLP}, though our assumptions on $\epsilon$ and $\mu$ are slightly different.

\begin{theorem}\label{Homogenization}
Let $\eps^n,\mu^n \in \mathcal{M}_\C(\alpha,\beta;\Omega)$ be such that
\begin{equation*}
\begin{cases}
\mu^n \xrightarrow{H} \mu^*,\\
\eps^n \xrightarrow{H} \eps^* 
\end{cases} \mbox{ in }\Omega, \mbox{ as }n \to \infty.
\end{equation*}
Suppose $(E^n,H^n) \in H(curl,\Omega)\times H(curl,\Omega)$ is the unique solution of the Maxwell system \eqref{newmaxwell}. Then, up to a subsequence,
\begin{equation*}
(E^n,H^n) \rightharpoonup (E,H) \mbox{ weakly in } H(curl,\Omega)\times H(curl,\Omega)\mbox{ as } n \to \infty.
\end{equation*}
where $(E,H) \in H(curl,\Omega)\times H(curl,\Omega)$ is the unique solution of the following homogenized time-harmonic Maxwell system:
\begin{equation}\label{hom}
\begin{cases}
\nabla \times E = i\omega \mu^* H\\
\nabla \times H = -i\omega \epsilon^* E \\
\nu \times E|_{\pa \Omega} = f
\end{cases}
\end{equation}
\end{theorem}

\begin{proof}
Consider the equation satisfied by $H^n$:
\begin{equation}\label{hn}
\nabla \times \frac{1}{i\omega}(\eps^n)^{-1}\nabla \times H^n +i\omega\mu^nH^n = 0. 
\end{equation}
By the estimate \eqref{esti}, we have $\|H^n\|_{H(curl,\Omega)} \leq C(\alpha,\beta,\Omega)$. Therefore, up to a subsequence, 
\begin{equation*}
H^n \rightharpoonup H \qquad \mbox{in } H(curl,\Omega).
\end{equation*}
Also, since $\|(\eps^n)^{-1}(\curl H^n)\|_{L^2} \leq C(\alpha,\beta,\Omega)$,
\begin{equation}\label{enhom}
(\eps^n)^{-1}\curl H^n \rightharpoonup \mathfrak{h}_1 \qquad \mbox{in } (L^2(\Omega))^3.
\end{equation}
Similarly, $\mu^nH^n$ is bounded in $L^2(\Omega)^3$. So assume
\begin{equation}\label{munhom}
\mu^n H^n \rightharpoonup \mathfrak{h}_2 \qquad \mbox{weakly in } L^2(\Omega)^3.
\end{equation}
Then, from the equation \eqref{hn}, we have
\begin{equation}\label{impid}
0 =\nabla \times \frac{1}{i\omega}(\eps^n)^{-1}\nabla \times H^n +i\omega\mu^nH^n \rightharpoonup \curl \frac{1}{i\omega}\mathfrak{h}_1 +i\omega \mathfrak{h}_2 \quad \mbox{ in } (H(curl,\Omega))'
\end{equation}
Now, let $u^n \in H^1(\Omega)$ solve
\begin{equation*}\begin{aligned}
-\nabla \cdot \left( \eps^n(x)\nabla u^n(x)\right) &= F \mbox{ in }\Omega\\
u^n &= 0 \mbox{ on }\partial\Omega
\end{aligned}
\end{equation*}
where $F\in H^{-1}(\Omega)$. Then we know that up to a subsequence, as $n \to \infty$, we have
\begin{align*}
u^n &\rightharpoonup u \quad\mbox{weakly in }H^1(\Omega)\\ 
\eps^n\nabla u^n &\rightharpoonup \eps^{*}\nabla u\quad\mbox{weakly in }L^2(\Omega)^3
\end{align*}
where $u\in H^{1}(\Omega)$ solves
\begin{equation}\label{32}
\begin{aligned}
-\nabla \cdot\left( \eps^{*}(x)\nabla u(x)\right) &= F \mbox{ in }\Omega\\
u &= 0 \mbox{ on }\partial\Omega.
\end{aligned}
\end{equation}
Let us consider the following identity:
\begin{equation}\label{33}
((\eps^n)^{-1}\curl H^n)\cdot \eps^n\nabla u^n = (\curl H^n)\cdot \nabla u^n \quad \mbox{in }\Omega.
\end{equation}
From \eqref{munhom} and \eqref{impid}, $(\eps^n)^{-1}\curl H^n \rightharpoonup \mathfrak{h}_1$ in $H(curl,\Omega)$. Also from \eqref{enhom} and \eqref{32}, $\eps^n \nabla u^n \rightharpoonup \eps^* \nabla u$ in $H(div,\Omega)$. Therefore, by the Div-Curl lemma, it follows that
\begin{equation*}
((\eps^n)^{-1}\curl H^n)\cdot \eps^n\nabla u^n \rightarrow \mathfrak{h}_1\cdot \eps^*\nabla u \quad \mbox{in } \mathcal{D}'(\Omega). 
\end{equation*}
Now, we find the limit of the right hand side of \eqref{33}. Since, $\nabla \cdot(\curl H^n)=0$ and $\curl(\nabla u^n) =0$,
we see that $\curl H^n \rightharpoonup \curl H$ in $H(div,\Omega)$ and $\nabla u^n \rightharpoonup \nabla u$ in $H(curl,\Omega)$. Therefore, again by the Div-Curl lemma, 
\begin{equation*}
\curl H^n \cdot \nabla u^n \to \curl H\cdot \nabla u \qquad \mbox{in }\mathcal{D}'(\Omega).
\end{equation*}
Now by equating the limits of the two sides of the equation, we get
\begin{equation*}
\mathfrak{h}_1\cdot \eps^* \nabla u = \curl H \cdot \nabla u \quad \mbox{in }\Omega.
\end{equation*}
In other words, $(\eps^*\mathfrak{h}_1-\curl H)\cdot \nabla u = 0$, where $u \in H^1_0(\Omega)$ is the solution of \eqref{32}. Since $(-\nabla \cdot(\eps^*\nabla))^{-1}: H^{-1}(\Omega) \to H^1_0(\Omega)$ is an isomorphism, by varying $F \in H^{-1}(\Omega)$, $u$ spans all of $H^1_0(\Omega)$. Therefore,
\begin{eqnarray*}
\eps^*\mathfrak{h}_1 &=& \curl H\\
\Rightarrow \mathfrak{h}_1 &=& (\eps^*)^{-1}\curl H.
\end{eqnarray*}
Next, we show that $\mathfrak{h}_2 = \mu^*H$. Let $w^n \in H^1(\Omega)$ be the unique solution of 
\begin{equation*}\begin{aligned}
-\nabla \cdot\left( \mu^{n}(x)\nabla w^n(x)\right) &= G \mbox{ in }\Omega\\
w^n &= 0 \mbox{ on }\partial\Omega
\end{aligned}
\end{equation*}
where $G \in H^{-1}(\Omega)$. Since $\mu^n \xrightarrow{H} \mu^*$,  up to a subsequence, as $n \to \infty$, we have
\begin{align*}
w^n &\rightharpoonup w \quad\mbox{weakly in }H^1(\Omega)\\ 
\mu^n\nabla w^n &\rightharpoonup \mu^{*}\nabla w\quad\mbox{weakly in }L^2(\Omega)^3
\end{align*}
where $w\in H^{1}(\Omega)$ solves
\begin{equation}\label{35}
\begin{aligned}
-\nabla \cdot\left( \mu^{*}(x)\nabla w(x)\right) &= G \mbox{ in }\Omega\\
w &= 0 \mbox{ on }\partial\Omega.
\end{aligned}
\end{equation}
Consider the identity
\begin{equation*}
 \mu^nH^n \cdot \nabla w^n = H^n \cdot \mu^n \nabla w^n. 
 \end{equation*}
We have $\nabla \cdot \mu^nH^n = 0$. So $\mu^nH^n \rightharpoonup \mathfrak{h}_2$ weakly in $H(div,\Omega)$. Similarly, since $\curl \nabla w^n = 0$, $\nabla w^n \rightharpoonup \nabla w$ weakly in $H(curl,\Omega)$. So, applying the Div-Curl lemma again, we get
\begin{equation*}
\mu^nH^n\cdot \nabla w^n \to \mathfrak{h}_2\cdot \nabla w \qquad \mbox{in } \mathcal{D}'(\Omega).
\end{equation*}
From the right hand side, we have $H^n \rightharpoonup H$ in $H(curl,\Omega)$ and that $\mu^n \nabla w^n \rightharpoonup \mu^* \nabla w$ in $H(div,\Omega)$ by \eqref{35}. So again by Div-Curl lemma, we have
\begin{equation*}
H^n \cdot \mu^n \nabla w^n \to H\cdot \mu^* \nabla w \qquad \mbox{in } \mathcal{D}'(\Omega).
\end{equation*}
Now, by equating the two limits, we get
\begin{eqnarray*}
\mathfrak{h}_2\cdot \nabla w &=& H\cdot \mu^* \nabla w\\
\Rightarrow (\mathfrak{h}_2-\mu^* H)\cdot \nabla w &=& 0.
\end{eqnarray*}
Again, by varying $G$, $w$ spans $H^1_0(\Omega)$. So we get $\mathfrak{h}_2 = \mu^*H$ in $\Omega$. Thus, we have the following homogenized equation for $H \in H(curl,\Omega)$:
\begin{equation*}
\curl \frac{1}{i\omega}(\eps^*)^{-1}\curl H +i\omega \mu^*H = 0 \qquad \mbox{in }H(curl,\Omega).
\end{equation*}
As before, set $E = \frac{i}{\omega}(\eps^*)^{-1}\curl H$. Then we see that $(E^n,H^n) \rightharpoonup (E,H)$ and $(E,H) \in H(curl,\Omega)\times H(curl,\Omega)$ satisfies \eqref{hom}.
\end{proof}

\subsection{Our goal : Main Theorem of the paper}
As we have seen, $(B_3\setminus \overline{B}_1;\eps^\rho_c,\mu^\rho_c)$ forms an approximate electromagnetic cloak for the region $B_1$. Therefore, for any $f \in H^{-1/2}(div,\pa B_3)$,
\[
 \Lambda_{\eps^\rho_e,\mu^\rho_e}f \rightarrow \Lambda_{I,I}f \quad \mbox{strongly in }H^{-1/2}(div,\pa B_3)\]
as $\rho \to 0$, where $(B_3;\eps^\rho_e,\mu^\rho_e)$ is the extended object  defined as before. We now want to construct \emph{isotropic} sequences $\{\eps_m\}$ and $\{\mu_m\}$ such that 
\[
\Lambda_{\eps_m,\mu_m}f \rightarrow \Lambda_{I,I}f \quad \mbox{strongly as } m \to \infty. \]
We would like to do this by first constructing sequences of permittivities and permeabilities whose $H$-limits are $\eps^\rho_e$ and $\mu^\rho_e$, and then passing to the limit as $\rho \to 0$. However, as we remarked in the introduction, we will not be able to directly find sequences whose homogenized limits are $(\eps^\rho_e,\mu^\rho_e)$ since $\mu^\rho_e \notin \mathcal{M}_{\C}(\alpha,\beta;B_3)$ for any $\alpha,\beta$. Therefore, we will go through a two-step process: first we fix a small parameter $\delta >0$, and  construct sequences $\eps^n_\delta, \mu^n$ such that
\begin{eqnarray*}
(1+i\delta)^{-1}\eps^n_\delta &\xrightarrow{H}& (1+i\delta)\eps^\rho_e,\\
(1+i\delta)\mu^n &\xrightarrow{H}& (1+i\delta)\mu^\rho_e.
\end{eqnarray*}
Note that both sides of the equations above are in $\mathcal{M}_{\C}(\alpha,\beta;B_3)$ for some $\alpha,\beta$, so that Theorem \ref{Homogenization} applies. In the next step, we let $\delta \to 0$. We will be able to show in Section 5 that
\[
\lim_{\delta \to 0} \left( \lim_{n\to \infty} \Lambda_{\eps^n_\delta,\mu^n}f\right) = \Lambda_{\eps^\rho_e,\mu^\rho_e}f \]
where the limit is in the strong topology of $H^{-1/2}(div, \pa B_3)$. We note however that the order in which we take the limits can not be interchanged. Finally, by a diagonal argument, by choosing sequences $n_m \to \infty, \delta_m \to 0$ and $\rho_m \to 0$, we can construct  $(\eps_m,\mu_m)$ such that
\[
\Lambda_{\eps_m,\mu_m}f \to \Lambda_{I,I}f \quad \mbox{strongly as }m \to \infty. \]

\section{Construction of the approximate cloak}

In this section we shall give an explicit construction of the approximate isotropic cloak. Throughout this section we assume that $\rho >0$ is a fixed parameter. Let us define 
\begin{equation*}
\gamma^* = \begin{cases}
(F_\rho)_*I & \mbox{in } B_3 \setminus B_{1/2},\\
1 & \mbox{in } B_{1/2},
\end{cases}
\end{equation*}
where $F_\rho$ is as defined in \eqref{frho}. More explictly,
\begin{equation}\label{gamma*}
\gamma^*(x) = \begin{cases}
I & \mbox{for } 2 \leq |x| \leq 3\\
\frac{1}{b}(I-\Pi(x)) +\frac{(|x|-a)^2}{b|x|^2}\Pi(x) & \mbox{for } 1< |x|< 2\\
\rho^{-1} I & \mbox{for } 1/2 < |x| < 1\\
I & \mbox{for } |x| <1/2
\end{cases} 
\end{equation}
where
\begin{equation*}
a = \frac{2(1-\rho)}{2-\rho}, \qquad b = \frac{1}{2-\rho}
\end{equation*}
and $\Pi(x) = |x|^{-2}xx^t$ is the projection in the radial direction. Given arbitrary regular $\eps_a, \mu_a$ in $B_{1/2}$, we note that the extended object is given by
\begin{equation*}
(\eps_e^\rho,\mu_e^\rho) = \begin{cases}
(\gamma^*,\gamma^*) & \mbox{in }B_3 \setminus B_1 \\
(\gamma^*(1+i\rho^{-2}/\omega),\gamma^*)& \mbox{in } B_1 \setminus B_{1/2} \\
(\eps_a \gamma^*, \mu_a \gamma^*)& \mbox{in } B_{1/2}.
\end{cases}
\end{equation*}
We will construct isotropic matrices of the form
\begin{equation}\label{ad7}
\gamma^n(x) = \gamma(x,n|x|)I ,\quad x\in B_3
\end{equation}
that $H$-converge to $\gamma^*$, where, $\gamma(x,y)$ is periodic in $y$ with period 1. Recall that $\gamma^*$ can be computed from $\gamma(x,y)$ using \eqref{localperiod}. From $\gamma^n$, it is straightforward to construct isotropic electric permittivities and magnetic permeabilities for the approximate cloak. Our construction is mostly based on \cite{GKLU20}. Let us change to polar coordinates. Let $s=(r,\theta,\varphi)$ and $t=(r^\prime,\theta^\prime,\varphi^\prime)$ be the spherical polar coordinates corresponding to the two scales $x$ and $y$ respectively. Next we homogenize $\gamma(s,t)$ in the $(r^\prime,\theta^\prime,\varphi^\prime)$-coordinates.
Let  $e_1,e_2,e_3$ denote the canonical basis vectors of $\mathbb{R}^3$  in $r^\prime$, $\theta^\prime$ and $\varphi^\prime$ directions respectively. Then for almost every $s\in\Omega $, there exist unique solutions $\chi_k(s,t),\, k=1,2,3$ of the equation
\begin{equation}\label{ad8}\begin{aligned}
&\nabla_{t}\cdot (\gamma(s,t)(\nabla_{t}\chi_k(s,t)+e_k))=0
\mbox{ in } \mathbb{R}^N,\\
 &t=(r^\prime,\theta^\prime\varphi^\prime) \mapsto \chi_k(s,t) \quad\mbox{is $1$-periodic in each of  $r^\prime,\theta^\prime,\varphi^\prime$. }
\end{aligned}\end{equation}
which satisfies the condition
\begin{equation}\label{ad9}
\int_ {Y} \chi_k(s,t)dt=0,
\end{equation}
where $dt =dr^\prime d\theta^\prime d\varphi^\prime$ and $Y=[0,1]^3$.

Since $\gamma(s,t)$ is independent of $\theta^\prime$ and $\varphi^\prime$, \eqref{ad8} and \eqref{ad9} imply that $\chi_k=0$ for $k=2,3$. Now consider the equation \eqref{ad8} for $\chi_1$:
\begin{equation}\label{ad11}
\frac{\p}{\p r^\prime} \left( \gamma(s, r^\prime)\frac{\p \chi_1(s,t)}{\p r^\prime}\right)=
-\frac{\p  \gamma(s, r^\prime)}{\p r^\prime}.
\end{equation}
It is clear from the above equation that $\chi_1$ is independent of $\theta^\prime,\varphi^\prime$ as well. Moreover, from \eqref{ad11} we get
\begin{equation*}
\frac{\p \chi_1}{\p r^\prime} =-1 + \frac{C}{\gamma(s,t)}
\end{equation*}
where the constant $C$ can be found by using the periodicity of $\chi_1$  with respect to $r^\prime$ to be
\[
C =\frac{1}{\int_0^1\gamma^{-1}(s,t)dr^\prime } :=\underline{\gamma}(s).
\]
Here $\underline{\gamma}(s)$ denotes the harmonic mean of $\gamma(s,t)$ in the second variable. Similarly, we let  $\overline{\gamma}(s)$ denote the arithmetic mean of $\gamma(s,t)$ in the second variable:
\[\overline{\gamma}(s) =\int_0^1 \gamma(s,t)dr^\prime.
\]
Now, for $\gamma^*$ to be the homogenized limit of $\gamma^n$, we must have
\[
 \gamma^*_{kl}(s) = \int_Y \gamma(s,r^\prime)\left(\frac{\partial\chi_k(s,t)}{\partial t_l} + \delta_{kl}\right)dt,
\]
which simplifies to
\begin{equation}\label{gamma}
\gamma^*(s)=\underline{\gamma}(s)\Pi(s)+\overline{\gamma}(s)(I-\Pi(s)),
\end{equation}
where, as before, $\Pi(s):\mathbb{R}^3\to \mathbb{R}^3$ is the projection on to the radial direction. Comparing this with \eqref{gamma*}, we see that it suffices to construct a $\gamma(s,t)$ such that
\begin{equation}\label{arith}
\overline{\gamma}(s) =
\begin{cases}
1 & \mbox{in } B_3 \setminus B_2\\
\frac{1}{b} & \mbox{in } B_2 \setminus B_1\\
\rho^{-1} & \mbox{in } B_1 \setminus B_{1/2}\\
1 & \mbox{in } B_{1/2} 
\end{cases}
\end{equation}
and
\begin{equation}\label{harm}
\underline{\gamma}(s) =
\begin{cases}
1 & \mbox{in } B_3 \setminus B_2\\
\frac{(|s|-a)^2}{b|s|^2}& \mbox{in } B_2 \setminus B_1\\
\rho^{-1} & \mbox{in } B_1 \setminus B_{1/2}\\
1& \mbox{in } B_{1/2}
\end{cases}
\end{equation}

Let us now construct such a $\gamma$. Our construction is based on the one presented in \cite{GKLU20}. However, we are able to work with a simpler construction since we do not need $\gamma$ to be continuous.  It is easy to define $\gamma(s,t)$ for $s \in B_1 \cup (B_3 \setminus B_2)$:
\begin{equation}\label{trivialgamma}
\gamma(s,t) =\begin{cases}
1 & \mbox{for } s \in B_3 \setminus B_2, t \in Y\\
\rho^{-1} & \mbox{for } s \in B_1\setminus B_{1/2}, \, t \in Y,\\
1 & \mbox{for} s \in B_{1/2}, \, t \in Y.
\end{cases}
\end{equation}
For $s \in B_2 \setminus B_1$, suppose we can write 
\begin{equation}\label{finalgamma}
\gamma(s,t) = \alpha(s)\chi_{(0,1/2)}(t) +\beta(s)\chi_{(1/2,1)}(t) 
\end{equation}
where $\alpha$ and $\beta$ are positive functions of $s$. Then \eqref{arith} and \eqref{harm} translate to
\begin{equation}\label{cond}
\begin{cases}
\overline{\gamma}(s) = \frac{\alpha(s)+\beta(s)}{2} = \frac{1}{b},\\
\underline{\gamma}(s) = \frac{2\alpha(s)\beta(s)}{\alpha(s)+\beta(s)} = \frac{(|s|-a)^2}{b|s|^2},
\end{cases}
\qquad \forall s \in B_2 \setminus B_1.
\end{equation}
Eliminating $\beta$ from the above equations, we get
\begin{equation*}
2\alpha^2-\frac{2}{b}\alpha +\frac{(|s|-a)^2}{b^2|s|^2} = 0. 
\end{equation*}
This equation will have two positive roots so long as the descriminant is non-negative:
\begin{eqnarray*}
\frac{4}{b^2}-\frac{8(|s|-a)^2}{b^2|s|^2} & \geq & 0\\
\Leftrightarrow 4|s|^2-16a|s|+8|a|^2 &\leq & 0
\end{eqnarray*}
which is equivalent to 
\[
(2-\sqrt{2})a \leq s \leq (2+\sqrt{2})a. \]
But this condition is true for any $0 < \rho <1/2$, as
\begin{eqnarray*}
(2-\sqrt{2})a &= & (2-\sqrt{2})\frac{2-2\rho}{2-\rho}< 2-\sqrt{2} < 1 \leq |s|, \quad \mbox{and }\\
(2+\sqrt{2})a &=& (2+\sqrt{2})\lb 1+\frac{\rho}{2-\rho}\rb > (2+\sqrt{2})\frac{2}{3} > 2 \geq |s|.
\end{eqnarray*}
Therefore, it follows that \eqref{cond} can be solved for $\alpha$ and $\beta$ to obtain positive functions of $s$ on $B_2 \setminus B_1$. Finally, equations \eqref{trivialgamma} and \eqref{finalgamma} define a $\gamma(s,t)$ with all the desired properties.\\

\subsection*{Construction of isotropic electromagnetic parameters :}

We shall now construct the electromagnetic parameters of our approximate isotropic cloak. We continue to assume  that the parameter $\rho>0$ introduced in the definition of the approximate cloak (cf. Definition \ref{appr})  is fixed. Recall that we defined \[\gamma^n(x) = \gamma(x,n|x|)\quad\mbox{ for all }x \in B_3, \quad n\in\mathbb{N}.\]
We define a sequence of isotropic non-singular  \textit{magnetic permeabilities} $\{\mu^n\}$ by
\begin{equation}\label{mn}
\mu^n := \varphi_1\gamma^n \quad \mbox{in }  B_3, 
\end{equation}
where 
\begin{equation*}
\varphi_1(x) = \begin{cases}
1& \mbox{in } B_3 \setminus B_{\frac{1}{2}}\\
\mu_a & \mbox{in } B_{\frac{1}{2}}
\end{cases}
\end{equation*}
and $\mu_a$ is an arbitrary permeability in $B_{\frac{1}{2}}$ as introduced in the Definition of cloaking (cf. Definition \ref{arb}). 

Next, we fix $\delta >0$ and define sequences of isotropic non-singular \textit{electric permittivities} $\{\widetilde{\epsilon}_{\delta}^{n}\}$ and isotropic  \textit{conductivities} $\{\sigma_{\delta}^{n}\}$ as follows:
\[
\widetilde{\epsilon}_{\delta}^{n} :=\mathcal{R}(\epsilon_{\delta}^{n})\quad\mbox{in }B_3 \quad\mbox{and}\quad \sigma^{n}_{\delta}:=\omega\Im(\epsilon_{\delta}^{n})\quad\mbox{in }B_3\]
where
\begin{equation}\label{end}
\epsilon_{\delta}^{n}=\widetilde{\epsilon}_{\delta}^{n} + \frac{i}{\omega}\sigma_{\delta}^{n} := (1+i\delta)^2\lb 1+\frac{i}{\omega}\varphi_3\rb \varphi_2\gamma^n \quad\mbox{in }B_3,
\end{equation}
and
\begin{equation*}
\varphi_2(x) = \begin{cases}
1& \mbox{in } B_3 \setminus B_{\frac{1}{2}}\\
\widetilde{\epsilon}_a+\frac{i}{\omega}\sigma_a & \mbox{in } B_{\frac{1}{2}}
\end{cases}
\end{equation*}
Here $\widetilde{\epsilon}_a,\sigma_a$ are arbitrary electric permittivity and conductivity respectively in $B_{\frac{1}{2}}$ as introduced in the definition of cloaking (cf. Definition \ref{arb}) and 
\[\varphi_3 =\rho^{-2}\chi_{B_1\setminus B_{\frac{1}{2}}}.\]

Now consider the following system of Maxwell equations:
\begin{equation*}
\begin{cases}
\nabla \times E^n_\delta = i\omega \mu^n H^n_\delta\quad\mbox{ in }B_3\\
\nabla \times H^n_\delta = -i\omega \epsilon^n_\delta E^n_\delta\quad\mbox{ in }B_3\\
\nu\times E^n_\delta =f \quad\mbox{ on }\partial B_3
\end{cases}
\end{equation*}
where $f \in H^{-1/2}(div,\pa B_3)$. $H^n_\delta$ satisfies the equation
\begin{equation}\label{n}
\frac{1}{i\omega}\nabla\times\left(  (\eps_{\delta}^n)^{-1}\nabla\times H^n_\delta\right) + i\omega \mu^n_\delta H^n_\delta = 0 \quad\mbox{in }B_3 .
\end{equation}
Multiplying throughout by $(1+i\delta)$, we get
\begin{equation*}
\frac{1}{i\omega}\nabla\times\left(  \left((1+i\delta)^{-1}\eps_{\delta}^n\right)^{-1}\nabla\times H^n_\delta\right) + i\omega (1+i\delta)\mu^n H^n_\delta = 0 \quad\mbox{in }B_3.
\end{equation*}
The above equation can be written in the variational form as
\begin{equation*}
a^n_\delta(H^n_\delta,w) = (1+i\delta)\int_{\pa B_3}f\cdot \overline{w} \qquad \forall w \in H(curl,B_3),
\end{equation*}
where $a^n_\delta$ is the sesquilinear form on $H(curl,B_3)$ defined by
\begin{equation*}
a^n_\delta(u,w) = -\int_{B_3} \left\{\frac{1}{i\omega}\lb(1+i\delta)^{-1}\eps^n_\delta)^{-1}\nabla \times u\rb\cdot (\nabla \times \overline{w}) +i\omega(1+i\delta)\mu^nu\cdot \overline{w}\right\}.
\end{equation*}
Let us define $\widehat{\mu}^n_\delta = (1+i\delta)\mu^n$. Then we see that 
\begin{eqnarray*}
\left(\widehat{\mu}^n_\delta-(\widehat{\mu}_{\delta}^n)^{\dagger}\right) &=& 2i\delta\varphi_1\gamma^n \\
\Rightarrow -\frac{i}{\omega}\xi\cdot \left(\widehat{\mu}_{\delta}^n-(\widehat{\mu}_{\delta}^n)^\dagger\right)\cdot \overline{\xi} &=&\frac{2\delta}{\omega}(\xi\varphi_1\gamma^n \overline{\xi}) > \delta c|\xi|^2. 
\end{eqnarray*}
for all $\xi \in \C^3$. Similarly, we define
\begin{equation*}
\widehat{\epsilon}^n_\delta =   (1+i\delta)^{-1}\eps^n_\delta
\end{equation*}
We find that 
\begin{eqnarray*}
\widehat{\eps}_{\delta}^n - (\widehat{\eps}_{\delta}^n)^{\dagger} &=& 2i \Im [(1+i\delta)^{-1}\eps^n_\delta]\\
&=& 2i\Im \left[ (1+i\delta)\left(1 +\frac{i}{\omega}\varphi_3\right)\varphi_2\gamma^n\right]\\
&=& 2i\left[\frac{\varphi_3}{\omega}+ \delta\right]\varphi_2\gamma^n\\
\Rightarrow -i\omega \xi \cdot (\widehat{\eps}_{\delta}^n - (\widehat{\eps}_{\delta}^n)^{\dagger})\cdot \overline{\xi} &=& 2\omega\left[\frac{\varphi_3}{\omega}\xi\varphi_2\cdot\gamma^n\overline{\xi} +\delta\xi\cdot \varphi_2\gamma^n \cdot \overline{\xi}\right]\\
&\geq & c\delta|\xi|^2
\end{eqnarray*}
where $c$ is some positive constant that depends on $\omega$. Note that $\widehat{\mu}_{\delta}^n$ and $\widehat{\eps}_{\delta}^n$ are uniformly bounded in $L^{\infty}(B_3)$ and we take $0 < \delta <1$. So all conditions for $H$-convergence are satisfied and we have
\begin{eqnarray*}
\widehat{\mu}_{\delta}^n & \xrightarrow{H} & (1+i\delta)\varphi_1 \gamma^* \quad \mbox{and} \\
\widehat{\eps}_{\delta}^n &\xrightarrow{H}& (1+i\delta)\lb 1+\frac{i}{\omega}\varphi_3\rb\varphi_2\gamma^* \quad \mbox{in } B_3. 
\end{eqnarray*}
by Proposition \ref{homprop}. We note that 
\begin{eqnarray*}
\mu^* &:=& \varphi_1 \gamma^* = \mu^\rho_e \quad \mbox{and}\\
\eps^* &:=& \lb 1+\frac{i}{\omega}\varphi_3\rb\varphi_2\gamma^* = \eps^\rho_e
\end{eqnarray*}
respectively, where $\eps^\rho_e,\mu^\rho_e$ are as in \eqref{rho}. Set
\begin{equation*}
\eps^*_\delta := (1+i\delta)^2\lb 1+\frac{i}{\omega}\varphi_3\rb\varphi_2\gamma^*.
\end{equation*}
We now apply Theorem \ref{Homogenization} to conclude that
\begin{eqnarray*}
H^n_\delta &\rightharpoonup & H_\delta \quad \mbox{and}\\
E^n_\delta &\rightharpoonup & E_\delta \quad\mbox{weakly in }H(curl,B_3)
\end{eqnarray*}
where $(H_\delta, E_{\delta}) \in H(curl,B_3)\times H(curl,B_3)$ is the solution of  the following homogenized equation corresponding to \eqref{n}:
\begin{equation*}
\begin{gathered}
\frac{1}{i\omega}\nabla\times\left(\left((1+i\delta)^{-1}(\eps^*_\delta)\right)^{-1}\nabla\times H_\delta\right) + i\omega(1+i\delta) \mu^* H_\delta = 0 \quad\mbox{in }B_3\\
E_\delta =\frac{-1}{i\omega}(\eps^*_\delta)^{-1}\nabla \times H_\delta\\
\nu\times E_\delta = f \quad\mbox{on }\partial B_3. 
\end{gathered}
\end{equation*}
Dividing the first equation by $1+i\delta$, we get
\begin{equation*}
\begin{gathered}
\frac{1}{i\omega}\nabla\times\left((\eps^*_\delta)^{-1}\nabla\times H_\delta\right) + i\omega \mu^* H_\delta = 0 \quad\mbox{in }B_3\\
E_\delta =\frac{-1}{i\omega}(\eps^*_\delta)^{-1}\nabla \times H_\delta\\
\nu\times E_\delta = f \quad\mbox{on }\partial B_3. 
\end{gathered}
\end{equation*}
Therefore, $(E_\delta,H_\delta)$ solve the following homogenized Maxwell system:
\begin{equation}\label{delta}
\begin{cases}
\curl E_{\delta} = i\omega \mu^*H_{\delta}\\
\curl H_{\delta} = -i\omega (\eps^*_\delta)E_{\delta}\\
\nu \times E_{\delta}|_{\partial B_3} = f \in H^{-1/2}(\mbox{div },\partial B_3)
\end{cases}
\end{equation}
Finally, we pass to the limit as $\delta \to 0$. Suppose $E,H$ satisfy
\begin{equation}\label{nodelta}
\begin{cases}
\curl E = i\omega \mu^* H \\
\curl H = -i\omega \eps^* E\\
\nu \times E|_{\partial B_3} = f.
\end{cases}
\end{equation}
Subtracting (\ref{nodelta}) from (\ref{delta}),
\begin{equation*}
\begin{cases}
\curl (E_\delta-E) =i\omega\mu^\rho_e(H_\delta-H)\\
\curl (H_\delta-H) = -i\omega\eps^\rho_e(E_\delta-E)+(2+i\delta)\omega \delta \eps^\rho_e E_\delta\\
\nu \times (E_\delta-E)|_{\pa B_3} = 0.
\end{cases}
\end{equation*}
Therefore, by Theorems \ref{existunique} and \ref{Nearcloak}, if $\omega$ is not an eigenvalue of $(B_3;I,I)$ and if $\rho>0$ is small enough, there exists a constant $C>0$ that is independent of $\delta$, but dependent on $\rho$, such that
\begin{equation*}
\|E_\delta-E\|_{H(curl,B_3)}+\|H_\delta-H\|_{H(curl,B_3)} \leq C\delta \|E_\delta\|_{L^2}.
\end{equation*}
Next, we write $E_\delta = (E_\delta-E)+E$ and apply the triangle inequality to get
\begin{eqnarray*}
\|E_\delta-E\|_{H(curl,B_3)}+\|H_\delta-H\|_{H(curl,B_3)} &\leq & C\delta(\|E_\delta-E\|_{L^2}+\|E\|_{L^2})\\
\Rightarrow (1-C\delta)\|E_\delta-E\|_{H(curl,B_3)}+\|H_\delta-H\|_{H(curl,B_3)} &\leq & C\delta\|E\|_{L^2}\\
\Rightarrow \|E_\delta-E\|_{H(curl,B_3)}+\|H_\delta-H\|_{H(curl,B_3)} &\leq & C'\delta\|E\|_{L^2}
\end{eqnarray*}
for small enough $\delta$. In conclusion, we have proved the following theorem:
\begin{theorem}\label{fieldconv} Suppose $\omega$ is not an electromagnetic eigenvalue for $(B_3;I,I)$ and suppose $\rho >0$ is small enough. Let $(E_\delta^nH_\delta^n), (E_\delta,H_\delta)$ be defined as above. Then for any fixed $\delta >0$,
\begin{equation*}
\lim_{n\to \infty}E^n_\delta = E_\delta, \quad \lim_{n \to \infty}H^n_\delta = H_\delta \quad \mbox{weakly in } H(curl,B_3)
\end{equation*}
and as $\delta \to 0$,
\begin{equation*}
\lim_{\delta\to 0}E_\delta = E, \quad \lim_{\delta \to 0}H_\delta = H \quad \mbox{strongly in }H(curl,B_3).
\end{equation*}
\end{theorem}

\section{Convergence of the Impedance map}

In this section we prove the main theorem of our paper, showing that our construction does in fact give an approximate isotropic cloak. Recall that we defined the impedance map $\Lambda_{\eps,\mu}:H^{-1/2}(div,\pa \Omega) \to H^{-1/2}(div, \pa \Omega)$ by setting $\Lambda_{\eps,\mu}f = \nu \times H|_{\pa \Omega}$, where $(E,H)$ is the unique solution of the boundary value problem
\begin{equation*}
\begin{cases}
\curl E = i\omega \mu H \quad \mbox{in } \Omega,\\
\curl H = -i\omega \eps E \quad \mbox{in } \Omega, \\
\nu \times E|_{\pa \Omega} = f.
\end{cases}
\end{equation*}
\begin{main}
Suppose $\omega > 0$ is not an eigenvalue for $(B_3;I,I)$, and let $\eps^n_\delta, \mu^n$ be defined as in the previous section. Then, given any $f \in H^{-1/2}(div,\pa B_3)$, we have
\begin{equation*}
\lim_{\delta \to 0}\lb \lim_{n \to \infty} \Lambda_{\eps^n_\delta, \mu^n}f \rb = \Lambda_{\eps^\rho_e,\mu^\rho_e}f \quad \mbox{strongly in } H^{-1/2}(div,\pa B_3),
\end{equation*}
and as a consequence
\[
\lim_{\rho \to 0}\lb \lim_{\delta \to 0}\lb \lim_{n \to \infty} \Lambda_{\eps^n_\delta, \mu^n}f \rb\rb = \Lambda_{I,I}f  \quad \mbox{strongly in } H^{-1/2}(div,\pa B_3).
\]
\end{main}
We emphasize once again that the order of the limits in the above equation can not be changed. Now, since $H_\delta \to H$ strongly in $H(curl,B_3)$, it follows immediately by the continuity of the trace map $H(curl,B_3) \to H^{-1/2}(div,\pa B_3)$ that
\begin{equation*}
\Lambda_{\eps^*_\delta,\mu^*}f = \nu \times H_\delta|_{\pa B_3} \to \nu \times H|_{\pa B_3} = \Lambda_{\eps^\rho_e,\mu^\rho_e} f
\end{equation*}
strongly in $H^{-1/2}(div,\pa B_3)$ as $\delta \to 0$. So it suffices to prove the following theorem:

\begin{theorem}\label{mainprop}
For a fixed $\delta >0$,
\begin{equation*}
 \Lambda_{\eps^n_\delta, \mu^n}f = \nu \times H^n_\delta|_{\pa B_3}  \to \nu \times H_\delta|_{\pa B_3} = \Lambda_{\eps^*_\delta,\mu^*}f
\end{equation*}
strongly in $H^{-1/2}(div,\pa B_3)$ as $n \to \infty$.
\end{theorem}

\begin{proof}
Let $\varphi \in C^{\infty}(\overline{B}_3)$ be a non-negative smooth function such that supp $\varphi \subset \overline{B}_3\setminus \overline{B}_2$ and $\varphi \equiv 1$ in a neighbourhood of $\pa B_3$. We note that on the support of $\varphi$,
\begin{eqnarray}
\eps^n_\delta &=& \eps^*_\delta = (1+i\delta)^2, \quad \mbox{and}\label{eps} \\ 
\mu^n &=& \mu^* = 1.\label{mu}
\end{eqnarray}
It is clear by the continuity of the trace map that
\begin{equation}\label{keyineq}
\|\nu \times (H^n_\delta-H_\delta)|_{\pa B_3}\|_{H^{-1/2}(div,\pa B_3)} \leq C(\|\varphi(H^n_\delta-H_\delta)\|_{L^2} +\|\curl (\varphi(H^n_\delta-H_\delta))\|_{L^2}).
\end{equation}
We will show that both the terms in the right hand side of the above equation will go to $0$. The proof will follow a series of steps.

\begin{step} For any $\psi \in C_c^\infty(B_3 \setminus \overline{B}_2)$, we claim that
\[
\int \psi|E^n_\delta-E_\delta|^2 \to 0, \quad \int \psi|H^n_\delta-H_\delta|^2 \to 0 \qquad \mbox{as }n \to \infty. \]
\end{step}
Indeed, we already know $E^n_\delta -E_\delta \rightharpoonup 0$ in $H(curl,B_3 \setminus \overline{B}_2)$. Consequently, we also have $E^n_\delta -E_\delta \rightharpoonup 0$ in $L^2(B_3 \setminus \overline{B}_2)$. Moreover, $\nabla \cdot(E^n_\delta -E_\delta) = \frac{i}{\omega}(1+i\delta)^{-2}\nabla \cdot \curl(H^n_\delta-H_\delta) =  0$ in $B_3\setminus \overline{B}_2$. Therefore, $E^n_\delta -E_\delta \in H(div, B_3\setminus \overline{B}_2)$ and for any $v \in H(div, B_3\setminus \overline{B}_2)$,
\[
\int (E^n_\delta-E_\delta)\cdot \overline{v} +\nabla \cdot(E^n_\delta-E_\delta)\nabla \cdot v \to 0  \]
which implies that $E^n_\delta -E_\delta \rightharpoonup 0$ in $H(div,B_3\setminus \overline{B}_2)$ as well. Therefore, by the Div-Curl lemma, $|E^n_\delta-E_\delta|^2 = (E^n_\delta-E_\delta)\cdot (E^n_\delta-E_\delta) \to 0$ in $\mathcal{D}'(B_3\setminus \overline{B}_2)$. By similar arguments, we also have $|H^n_\delta-H_\delta|^2 \to 0$ in $\mathcal{D}'(B_3\setminus \overline{B}_2)$. This completes the proof of Step 1.

\begin{step} Next, we want to show that $\int \varphi^2|E^n_\delta-E_\delta|^2 \to 0$.
\end{step}
Consider the space of functions
\[
V = \{ u \in H(curl,B_3) : \nabla \cdot u \in L^2(B_3), \nu \times u|_{\pa B_3} = 0\} \]
with the following inner product:
\begin{eqnarray*}
\langle u,w\rangle_V &=& \langle u,w\rangle_{H(curl,B_3)} +\langle \nabla \cdot u,\nabla \cdot w\rangle_{L^2}\\
&=& \langle u,w\rangle_{L^2}+\langle \curl u,\curl w\rangle_{L^2}+\langle \nabla \cdot u,\nabla \cdot w\rangle_{L^2}. 
\end{eqnarray*}
It is well known $V$ is a Hilbert space and that the inclusion $V \hookrightarrow L^2(B_3)^3$ is compact \cite{Nedelec}. Now, we know that $\nu \times \varphi(E^n_\delta-E_\delta)|_{\pa B_3} = 0$ and $\varphi(E^n_\delta-E_\delta) \rightharpoonup 0$ in $H(curl,B_3)$. Next, consider
\[
\nabla \cdot (\varphi(E^n_\delta-E_\delta)) = \varphi \nabla \cdot (E^n_\delta-E_\delta) + \nabla \varphi \cdot (E^n_\delta-E_\delta). \]
The first term on the right hand side vanishes identically since $\nabla \cdot (E^n_\delta-E_\delta) = 0$ on supp $\varphi$. Also, since $\varphi \equiv 1$ in a neighbourhood of $\pa B_3$, $\nabla \varphi \in C_c^{\infty}(B_3 \setminus \overline{B}_2)$. Therefore, from Step 1, $\nabla \varphi \cdot (E^n_\delta-E_\delta) \to 0$ strongly in $L^2(\Omega)$. As a consequence, for any $w \in V$,
\[
\langle \varphi(E^n_\delta-E_\delta), w\rangle_V =  \langle \varphi(E^n_\delta-E_\delta),w\rangle_{H(curl,B_3)} + \langle  \nabla \cdot (\varphi(E^n_\delta-E_\delta)),\nabla \cdot w\rangle_{L^2} \to  0 \]
as $n \to \infty$. Therefore, $\varphi(E^n_\delta-E_\delta) \rightharpoonup 0$ in $V$ as well. By the compactness of the inclusion $V \hookrightarrow L^2(B_3)^3$, this further implies $\varphi(E^n_\delta-E_\delta) \to 0$ strongly in $L^2(B_3)^3$.

\begin{step} We now show that $\int |\curl (\varphi(H^n_\delta-H_\delta)|^2 \to $0.
\end{step}
Note that
\[
\curl(\varphi(H^n_\delta-H_\delta)) = \nabla\varphi \times (H^n_\delta-H_\delta) + \varphi \curl(H^n_\delta-H_\delta). \]
As we have already observed, $\nabla \varphi \in C_c^\infty(B_3)$ and therefore by Step 1, the first term on the right hand side converges strongly to $0$ in $L^2(B_3)$. On the other hand, $\curl(H^n_\delta-H_\delta) =-\frac{i}{\omega}(E^n_\delta-E_\delta)$ on supp $\varphi$. Therefore, by Step 2, the second term on the right hand side also converges to $0$ strongly in $L^2$. This completes Step 3.

\begin{step} The final step is to show that $\int \varphi^2|H^n_\delta-H_\delta|^2 \to 0$.
\end{step}
Consider the variational equations satisfied by $H^n_\delta$ and $H_\delta$: for all $w \in H(curl,\Omega)$,
\begin{eqnarray*}
-\int \frac{1}{i\omega}\{(\eps^n_\delta)^{-1}\curl H^n_\delta\}\cdot \curl \overline{w} +i\omega\mu^nH^n_\delta \cdot \overline{w}&=& \int_{\pa \Omega}f\cdot \overline{w}, \quad \mbox{and}\\
-\int \frac{1}{i\omega}\{(\eps^*_\delta)^{-1}\curl H_\delta\}\cdot \curl \overline{w} +i\omega\mu^*H_\delta\cdot \overline{w}&=& \int_{\pa \Omega}f\cdot \overline{w}.
\end{eqnarray*}
Taking the difference of these two equations and rearranging terms, we get
\begin{equation*}
\int (\mu^nH^n_\delta -\mu^*H_\delta)\cdot \overline{w} = \frac{1}{\omega^2}\int \{(\eps^n_\delta)^{-1}\curl H^n_\delta - (\eps^*_\delta)^{-1}\curl H_\delta\}\cdot \curl \overline{w}.
\end{equation*}
Now let $w = \varphi^2(H^n_\delta-H_\delta)$. By equations (\ref{eps}) and (\ref{mu}),
\begin{eqnarray*}
\int \varphi^2|H^n_\delta-H_\delta|^2 &=& \omega^{-2}(1+i\delta)^{-2}\int \curl(H^n_\delta-H_\delta)\cdot \curl(\varphi(H^n_\delta-H_\delta)).\\
\Rightarrow \|\varphi(H^n_\delta-H_\delta)\|_{L^2}^2 &\leq & C\|H^n_\delta-H_\delta\|_{H(curl,\Omega)}\|\curl(\varphi(H^n_\delta-H_\delta))\|_{L^2}.
\end{eqnarray*}
It is clear from Step 3 that the right hand side of the above equation goes to $0$ as $n \to \infty$. Therefore, we conclude that $\varphi(H^n_\delta-H_\delta) \to 0$ strongly in $L^2(B_3)^3$. Finally, applying the conclusions of Steps 3 and 4 to (\ref{keyineq}) gives us the desired result.

\end{proof}

\section*{Acknowledgments}

The authors would like to thank Gunther Uhlmann for suggesting the problem. Ashwin Tarikere was partially supported by the National Science Foundation under grant DMS-1265958 and a Visiting Scholarship at the Institute for Avanced Study, Hong Kong University of Science and Technology.

\bibliography{Master_bibfile}

\def\cprime{$'$}
\begin{thebibliography}{10}

\bibitem{A}
Gr{\'e}goire Allaire.
\newblock {\em Shape optimization by the homogenization method}, volume 146 of
  {\em Applied Mathematical Sciences}.
\newblock Springer-Verlag, New York, 2002.

\bibitem{Ammari2}
Habib Ammari, Josselin Garnier, Vincent Jugnon, Hyeonbae Kang, Hyundae Lee, and
  Mikyoung Lim.
\newblock Enhancement of near-cloaking. part iii: Numerical simulations,
  statistical stability, and related questions.
\newblock {\em Contemporary Mathematics}, 577:1--24, 2012.

\bibitem{Ammari1}
Habib Ammari, Hyeonbae Kang, Hyundae Lee, and Mikyoung Lim.
\newblock Enhancement of near-cloaking. part ii: The helmholtz equation.
\newblock {\em Communications in Mathematical Physics}, 317(2):485--502, 2013.

\bibitem{Ammari3}
Habib Ammari, Hyeonbae Kang, Hyundae Lee, Mikyoung Lim, and Sanghyeon Yu.
\newblock Enhancement of near-cloaking for the full maxwell equations.
\newblock {\em SIAM J. Appl. Math}, 73(6).

\bibitem{NearCloak}
Gang Bao and Hongyu Liu.
\newblock Nearly cloaking the electromagnetic fields.
\newblock {\em SIAM J. Appl. Math.}, 74(3):724--742, 2014.

\bibitem{NearCloak2}
Gang Bao, Hongyu Liu, and Jun Zou.
\newblock Nearly cloaking the full maxwell equations: Cloaking active contents
  with general conducting layers.
\newblock {\em Journal de Mathématiques Pures et Appliquées}, 101(5):716 --
  733, 2014.

\bibitem{BLP}
Alain Bensoussan, Jacques-Louis Lions, and George Papanicolaou.
\newblock {\em Asymptotic analysis for periodic structures}, volume~5 of {\em
  Studies in Mathematics and its Applications}.
\newblock North-Holland Publishing Co., Amsterdam-New York, 1978.

\bibitem{BUF2}
A.~Buffa, M.~Costabel, and D.~Sheen.
\newblock On traces for {${\bf H}({\bf curl},\Omega)$} in {L}ipschitz domains.
\newblock {\em J. Math. Anal. Appl.}, 276(2):845--867, 2002.

\bibitem{BUF1}
Annalisa Buffa.
\newblock Trace theorems on non-smooth boundaries for functional spaces related
  to {M}axwell equations: an overview.
\newblock In {\em Computational electromagnetics ({K}iel, 2001)}, volume~28 of
  {\em Lect. Notes Comput. Sci. Eng.}, pages 23--34. Springer, Berlin, 2003.

\bibitem{DengLiuUhlmann}
Youjun Deng, Hongyu Liu, and Gunther Uhlmann.
\newblock Full and partial cloaking in electromagnetic scattering.
\newblock {\em Archive for Rational Mechanics and Analysis}, 223(1):265--299,
  2017.

\bibitem{Faraco}
Daniel Faraco, Yaroslav Kurylev, and Alberto Ruiz.
\newblock G-convergence, dirichlet to neumann maps and invisibility.
\newblock {\em Journal of Functional Analysis}, 267(7):2478 -- 2506, 2014.

\bibitem{TK}
Tuhin {Ghosh} and Karthik {Iyer}.
\newblock Cloaking for a quasi-linear elliptic partial differential equation.
\newblock {\em ArXiv e-prints}, MAR 2017.

\bibitem{GR}
Vivette Girault and Pierre-Arnaud Raviart.
\newblock {\em Finite element methods for {N}avier-{S}tokes equations},
  volume~5 of {\em Springer Series in Computational Mathematics}.
\newblock Springer-Verlag, Berlin, 1986.
\newblock Theory and algorithms.

\bibitem{FullWave}
Allan {Greenleaf}, Yaroslav {Kurylev}, Matti {Lassas}, and Gunther {Uhlmann}.
\newblock {Full-Wave Invisibility of Active Devices at All Frequencies}.
\newblock {\em Communications in Mathematical Physics}, 275:749--789, November
  2007.

\bibitem{Wormhole}
Allan {Greenleaf}, Yaroslav {Kurylev}, Matti {Lassas}, and Gunther {Uhlmann}.
\newblock {Electromagnetic Wormholes via Handlebody Constructions}.
\newblock {\em Communications in Mathematical Physics}, 281:369--385, July
  2008.

\bibitem{GKLU20}
Allan Greenleaf, Yaroslav Kurylev, Matti Lassas, and Gunther Uhlmann.
\newblock Isotropic transformation optics: approximate acoustic and quantum
  cloaking.
\newblock {\em New Journal of Physics}, 10(11):115024, 2008.

\bibitem{GKLU2}
Allan Greenleaf, Yaroslav Kurylev, Matti Lassas, and Gunther Uhlmann.
\newblock Cloaking devices, electromagnetic wormholes, and transformation
  optics.
\newblock {\em SIAM Rev.}, 51(1):3--33, 2009.

\bibitem{GKLU3}
Allan Greenleaf, Yaroslav Kurylev, Matti Lassas, and Gunther Uhlmann.
\newblock Invisibility and inverse problems.
\newblock {\em Bull. Amer. Math. Soc. (N.S.)}, 46(1):55--97, 2009.

\bibitem{IsoTechnical}
Allan {Greenleaf}, Yaroslav {Kurylev}, Matti {Lassas}, and Gunther {Uhlmann}.
\newblock Approximate quantum and acoustic cloaking.
\newblock {\em J. Spectral Theory}, 1(1):27--80, 2011.

\bibitem{GKLU0}
Allan {Greenleaf}, Matti {Lassas}, and Gunther {Uhlmann}.
\newblock Anisotropic conductivities that cannot be detected by eit.
\newblock {\em Physiological measurement}, 24(2):413, 2003.

\bibitem{GKLU1}
Allan Greenleaf, Matti Lassas, and Gunther Uhlmann.
\newblock On nonuniqueness for {C}alder\'on's inverse problem.
\newblock {\em Math. Res. Lett.}, 10(5-6):685--693, 2003.

\bibitem{Hetmaniuk}
Ulrich {Hetmaniuk} and Hongyu {Liu}.
\newblock {On acoustic cloaking devices by transformation media and their
  simulation}.
\newblock {\em SIAM J. Appl. Math}, 70(8):2996--3021, August 2010.

\bibitem{Hongyu}
Guanghui Hu and Hongyu Liu.
\newblock Nearly cloaking the elastic wave fields.
\newblock {\em J. Math. Pures Appl. (9)}, 104(6):1045--1074, 2015.

\bibitem{OL}
V.~V. Jikov, S.~M. Kozlov, and O.~A. Ole{\u\i}nik.
\newblock {\em Homogenization of differential operators and integral
  functionals}.
\newblock Springer-Verlag, Berlin, 1994.
\newblock Translated from the Russian by G. A. Yosifian [G. A.
  Iosif{\cprime}yan].

\bibitem{Kirsch}
Andreas Kirsch and Frank Hettlich.
\newblock {\em The mathematical theory of time-harmonic {M}axwell's equations},
  volume 190 of {\em Applied Mathematical Sciences}.
\newblock Springer, Cham, 2015.
\newblock Expansion-, integral-, and variational methods.

\bibitem{KLS}
Ilker {Kocyigit}, Hongyu {Liu}, and Hongpeng {Sun}.
\newblock {Regular scattering patterns from near-cloaking devices and their
  implications for invisibility cloaking}.
\newblock {\em Inverse Problems}, 29(4):045005, March 2013.

\bibitem{KOVW}
Robert Kohn, Daniel Onofrei, Michael Vogelius, and Michael Weinstein.
\newblock Cloaking via change of variables for the helmholtz equation.
\newblock {\em Communications on Pure and Applied Mathematics},
  63(8):973--1016, 2010.

\bibitem{KSVW}
Robert Kohn, Haiping Shen, Michael Vogelius, and Michael Weinstein.
\newblock Cloaking via change of variables in electric impedance tomography.
\newblock {\em Inverse Problems}, 24(1):015016, 21, 2008.

\bibitem{Leon}
Ulf Leonhardt.
\newblock Optical conformal mapping.
\newblock {\em Science}, 312(5781):1777--1780, 2006.

\bibitem{Carpet}
Jensen {Li} and John {Pendry}.
\newblock {Hiding under the Carpet: A New Strategy for Cloaking}.
\newblock {\em Physical Review Letters}, 101(20):203901, November 2008.

\bibitem{LiuHelmholtz}
Jingzhi Li, Hongyu Liu, Luca Rondi, and Gunther Uhlmann.
\newblock Regularized transformation-optics cloaking for the helmholtz
  equation: From partial cloak to full cloak.
\newblock {\em Communications in Mathematical Physics}, 335(2):671--712, 2015.

\bibitem{FSH}
Jingzhi Li, Hongyu Liu, and Hongpeng Sun.
\newblock Enhanced approximate cloaking by sh and fsh lining.
\newblock {\em Inverse Problems}, 28(7):075011, 2012.

\bibitem{Lin}
Yi-Hsuan {Lin}.
\newblock {Nearly cloaking for the elasticity system with residual stress}.
\newblock {\em ArXiv e-prints}, November 2016.

\bibitem{LiuScatter}
Hongyu {Liu}.
\newblock {Virtual reshaping and invisibility in obstacle scattering}.
\newblock {\em Inverse Problems}, 25(4):045006, April 2009.

\bibitem{LiuAcoustic}
Hongyu Liu.
\newblock On near-cloak in acoustic scattering.
\newblock {\em Journal of Differential Equations}, 254(3):1230 -- 1246, 2013.

\bibitem{FSH2}
Hongyu Liu and Hongpeng Sun.
\newblock Enhanced near-cloak by fsh lining.
\newblock {\em Journal de Mathématiques Pures et Appliquées}, 99(1):17 -- 42,
  2013.

\bibitem{LHUG}
Hongyu Liu and Gunther Uhlmann.
\newblock Regularized transformation-optics cloaking in acoustic and
  electromagnetic scattering.
\newblock In {\em Inverse problems and imaging}, volume~44 of {\em Panor.
  Synth\`eses}, pages 111--136. Soc. Math. France, Paris, 2015.

\bibitem{Zhou}
Hongyu Liu and Ting Zhou.
\newblock On approximate electromagnetic cloaking by transformation media.
\newblock {\em SIAM J. Appl. Math.}, 71(1):218--241, 2011.

\bibitem{Ting2}
Hongyu {Liu} and Ting {Zhou}.
\newblock {Two dimensional invisibility cloaking by transformation optics}.
\newblock {\em Discrete Contin. Dyn. Syst.}, 31:525--543, August 2011.

\bibitem{Monk}
Peter Monk.
\newblock {\em Finite element methods for {M}axwell's equations}.
\newblock Numerical Mathematics and Scientific Computation. Oxford University
  Press, New York, 2003.

\bibitem{Nedelec}
Jean-Claude N\'ed\'elec.
\newblock {\em Acoustic and electromagnetic equations}, volume 144 of {\em
  Applied Mathematical Sciences}.
\newblock Springer-Verlag, New York, 2001.
\newblock Integral representations for harmonic problems.

\bibitem{Nguyen1}
Hoai-Minh Nguyen.
\newblock Cloaking via change of variables for the helmholtz equation in the
  whole space.
\newblock {\em Communications on Pure and Applied Mathematics},
  63(11):1505--1524, 2010.

\bibitem{Nguyen2}
Hoai-Minh Nguyen and Michael~S. Vogelius.
\newblock Full range scattering estimates and their application to cloaking.
\newblock {\em Archive for Rational Mechanics and Analysis}, 203(3):769--807,
  2012.

\bibitem{PSS}
John Pendry, David Schurig, and David Smith.
\newblock Controlling electromagnetic fields.
\newblock {\em Science}, 312(5781):1780--1782, 2006.

\bibitem{T}
Luc Tartar.
\newblock {\em The general theory of homogenization}, volume~7 of {\em Lecture
  Notes of the Unione Matematica Italiana}.
\newblock Springer-Verlag, Berlin; UMI, Bologna, 2009.
\newblock A personalized introduction.

\end{thebibliography}
\bibliographystyle{plain}
\end{document}